\documentclass[12 pt]{amsart}
\usepackage{array, tabularx, amsmath, tikz, hyperref, amssymb}
\usetikzlibrary{matrix,arrows.meta, arrows}

\usepackage{fancyhdr}

\newtheorem{theorem}{Theorem}[section]
\newtheorem{lemma}[theorem]{Lemma}
\newtheorem{proposition}[theorem]{Proposition}
\newtheorem{corollary}[theorem]{Corollary}

\newtheorem*{introthmA}{Theorem A}
\newtheorem*{introthmB}{Theorem B}
\newtheorem*{introthmC}{Corollary B.0}

\newtheorem*{appthmARunge}{Theorem [Runge]}

\theoremstyle{definition}

\theoremstyle{remark}

\newtheorem{e.g.}[theorem]{Example}

\newcommand{\ora}{\overrightarrow}
\begin{document}

\title{Period polynomials for Picard modular forms}
\author{Sheldon Joyner}
\address{Department of Mathematics and Computer Science
\newline
Suffolk University
\newline
8 Ashburton Pl
\newline
Boston
\newline
MA 02108
\newline
{\tt{stjoyner@suffolk.edu}}
}


\begin{abstract}
The relations satisfied by period polynomials associated to modular forms yield a way to count dimensions of spaces of cusp forms. After showing how these relations arise from those on the mapping class group $PSL(2, \mathbb{Z})$ of the moduli space $\mathcal{M}_{0,4}$ of genus 0 curves with 4 marked points, the author goes on to define period polynomials associated to Picard modular forms. Relations on these Picard period polynomials are then determined, and via an embedding of a monodromy representation of the moduli space $\mathcal{M}_{0,5}$ of genus 0 curves with 5 marked points in $PU(2,1 ; \mathbb{Z}[\rho])$ (where $\rho$ denotes a third root of unity), they are  related to the geometry of $\mathcal{M}_{0,5}.$
\end{abstract}

\maketitle

\tableofcontents

\section*{Introduction}

In this paper, period polynomials associated to Picard modular forms are defined, by analogy with the usual period polynomials associated to modular forms. Relations satisfied by these Picard period polynomials are given and linked to the geometry of the moduli space $\mathcal{M}_{0,5}$ of genus 0 curves with 5 marked points. 


Period polynomials allow for the determination of the dimension of a given space of cusp forms, by providing an isomorphism thereof into an explicit Euclidean space.  This is the well-known Eichler-Shimura-Manin isomorphism as discussed for example in \cite{LangIntroModForms}, and finding the relevant Euclidean space is tantamount to determining relations on the period polynomials. As Brown observed in \cite{BrownPerPolys}, one of these relations is reminiscent of the hexagonal relation on the Grothendieck-Teichm\"{u}ller group $GT$. In fact, this three term symmetry relation satisfied by the period polynomials for modular forms \cite{GanglKZ} is the same as the logarithmic (Lie algebra) version of the hexagonal relation satisfied by associators \cite{Bar-NatanDancso}. Of course this is not a coincidence. It may be explained using a simple geometric argument based upon an explicit realization of 
$PSL(2, \mathbb{Z})$ 
as a group of paths in $\mathbb{P}^1\backslash\{0,1,\infty\}$ developed by the author in \cite{Joyner:qm} (see also \S 2.5 of  Schneps' article in \cite{Buffetal:99}), as shown in \S\ref{s:1} below. The deeper story is that the group-like elements of the image of the Betti-de Rham comparison isomorphism for unipotent completions of the topological and algebraic fundamental groups of $\mathbb{P}^1\backslash\{0,1,\infty\}$, (cf. \cite{Deligne:89}), when suitably interpreted may be identified with elements of $GT$ (see also \S 4 of Kedlaya's unpublished notes \cite{KedlayaDeligne} in conjunction with \cite{Joyner:qm}). The association of the Eichler-Shimura-Manin theory of $PSL(2, \mathbb{Z})$ with the Grothendieck-Teichm\"uller group that arises in this way is thus an echo of the work of  Brown and Hain in realizing a version of the Eichler-Shimura-Manin isomorphism as a special case of the Betti-de Rham comparison isomorphism over the moduli space $\mathcal{M}_{1,1}$ of elliptic curves - cf. \cite{BrownHainDeRham}, in which a cohomological construction using a period polynomial associated to a weakly modular form yields the comparison isomorphism in a particular case.

The reason that the hexagonal relation in $GT$ is linked to the geometry of $\mathbb{P}^1\backslash\{0,1,\infty\}$ has to do with the fact that the latter may be identified with the moduli space $\mathcal{M}_{0,4}$ of genus 0 Riemann surfaces with 4 marked points, since $GT$ is in fact the automorphism group of the tower of (profinite completions of) the fundamental groupoids $\mathcal{T}_{0,n}$ of the respective moduli stacks $\mathcal{M}_{0,n}$ of genus 0 surfaces with $n$ marked points (cf. \cite{Drinfel'd}). It is known that all relations in this tower come from the $n=4$ and $n=5$ cases. This begs the question of whether a version of Eichler-Shimura-Manin over a suitable explicit group $G$ of paths in $\mathcal{M}_{0,5}$ could be developed. The purpose of this paper 
is to present a first step in this direction, in providing the definition of an analogue of the period polynomials in this situation, investigating their symmetry relations, and relating the latter to the geometry of $\mathcal{M}_{0,5}$. 

On the De Rham side of the story, one reason that the moduli space $\mathcal{M}_{0,4}$ 
appears 
 is that it serves as a parameter space for elliptic curves via the Legendre form of the Weierstrass equation. The moduli space $\mathcal{M}_{0,5}$ on the other hand serves as a parameter space for Picard curves (cf. \cite{Shiga}) and for certain hyperelliptic curves of genus 3 (cf. \cite{KMatsumoto}) in a similar way. These two cases correspond to the situations in which the centralizer $\Gamma_M$ of a corresponding element $M$ of finite order in the modular group $Sp(3, \mathbf{Z})$ is unitary (as is discussed in \cite{RungePicardMod}). In the Picard curve case it turns out that $\Gamma_M \simeq PU(2, 1; \mathbb{Z}[\rho])$ consists of the matrices with entries in $\mathbb{Z}[\rho]$ in the projectivization of a unitary group of $3\times 3$ matrices preserving a certain Hermitian form and $\rho$ denotes the third root of unity $\exp(2 \pi i/3).$ Now in \cite{FalbelParkerPicard}, Falbel and Parker  
provided explicit presentations for $PU(2,1;\mathbb{Z}[\rho])$ along with a fundamental domain in the universal cover. Together with the geometric analysis of Shiga in \cite{Shiga} and Runge's definition of Picard modular forms for $PU(2,1;\mathbb{Z}[\rho])$ in \cite{RungePicardMod}, this work serves as a crucial ingredient in the definition given here of the Picard period polynomials. 




The organization of the paper is as follows: \S \ref{s:1} begins by assembling known facts relating to the monodromy representation for $\mathbb{P}^1\backslash\{0,1,\infty\}$ coming from the period mapping for elliptic curves, and the associated explicit realization of the mapping class group as a space of paths on $\mathbb{P}^1\backslash\{0,1,\infty\}$ given in \cite{Joyner:qm} is explained. Next, the usual period polynomials $P_f(X_0,X_1)$ are defined and the (known) symmetry relations they satisfy are proven entirely from consideration of the geometry of 
$\mathbb{P}^1\backslash\{0,1,\infty\}$. Indeed, if $\overrightarrow{01}$ denotes a tangential basepoint (see \S \ref{s:1.1} for the definition) at 0 in the direction of 1, $s$ is the tangential path from $\ora{01}$ to $\ora{10}$ and $t$ denotes a loop in the upper half plane in $\mathbb{P}^1\backslash\{0,1,\infty\}$ from $\ora{01}$ to $\ora{0\infty}$, and concatenation of paths occurs via a twisting by an $S_3$-action (again see \S \ref{s:1.1} below for the details); then we prove:
\begin{introthmA}
 Every contractible path in $\pi_1(\mathbb{P}^1\backslash\{0,1,\infty\}, \overrightarrow{01})$ gives a relation on $P_f(X_0,X_1)$. In particular, 
\begin{equation}
\label{e:DA1intro}
P_f(X_0,X_1) = -P_f(-X_1, X_0)
\end{equation}
corresponds to the path $s \circ s = 1$ and 
\begin{equation}
\label{e:DA2intro}
P_f(X_0,X_1) + P_f(-X_0-X_1, X_0) + P_f(X_1, -X_0-X_1) = 0
\end{equation}
corresponds to the contractibility of the hexagonal path $(s\circ t)^3$.
\end{introthmA}

In \S \ref{s:2}, period polynomials associated to Picard modular forms are defined. These forms are modular for the group $PU(2,1; \mathbb{Z}[\rho])$, into which a (non-faithful) monodromy representation for $\mathcal{M}_{0,5}$ associated to the period mapping for Picard curves is known to embed (cf. \cite{Shiga}). In an effort to emulate the extension of the monodromy representation in the $\mathcal{M}_{0,4}$ situation a surjective homomorphism of a group $\mathcal{G}_{(01,01)}^{(4)}$ of certain explicit paths in $\mathcal{M}_{0,5}$ onto $PU(2,1;\mathbb{Z}[\rho])$ is then given. 

The Picard period polynomial associated to the Picard modular form $f$ of Picard type $M$ (in terminology of Runge - see \S \ref{s:2.2} below) is next defined, by
\[
P_f(X_0,X_1,X_2)
= 
\int_D f(z_1,z_2)(z_1X_0+z_2X_1+X_2)^{3k-3}dz_1\wedge dz_2
\]
where $D$ is a certain 2-simplex in the two complex variables $z_1$ and $z_2$ on the Siegel domain (complex hyperbolic space)
\[
\mathbb{H}_{\mathbb{C}}^2
=
\{[z_1:z_2:1] \in \mathbb{P}^2(\mathbb{C})\;:\;
2\text{ Re }(z_1)+|z_2|^2 <0\}
\]
and $X_0, X_1$ and $X_2$ are formal variables. The 2-simplex $D$ is defined by analogy with the path in the upper half space over which the usual (elliptic) period polynomials are defined. After establishing the non-trivial fact that the integral converges, we go on to prove:
\begin{introthmB}
The period polynomial $P_f(X_0,X_1,X_2)$ for the Picard modular form $f$ of type $M$ and weight $k$ satisfies the following relations, each of which corresponds to a  defining relation for a certain presentation of  $PU(2,1;\mathbb{Z}[\rho]):$
\begin{equation}
\label{e:R^2intro}
P_f(X_0,X_1,X_2) = -P_f(X_2,-X_1,X_0)
\end{equation}

\begin{multline}
P_f(X_0,X_1,X_2)+(-\rho^2)^{k-1}P_f(X_0,-\rho X_1,X_2)+\rho^{k-1} P_f(X_0,\rho^2 X_1,X_2)
\\
\label{e:r1^6intro}
 +(-1)^{k-1}P_f(X_0,-X_1,X_2)+\rho^{2k-2}P_f(X_0,\rho X_1,X_2)+(-\rho)^{k-1}P_f(X_0,-\rho^2 X_1,X_2) =0
\end{multline}

\begin{multline}
\label{e:PMintro}
P_f(X_0,X_1,X_2)+\rho^{k-1}P_f(\rho^2 X_0+X_1+X_2, \rho^2 X_0-\rho^2 X_1, X_0) 
\\
 + \rho^{2k-2}P_f(\rho^2 X_2,\rho^2 X_2 - X_1, \rho^2 X_0+X_1+X_2) =0
\end{multline}


and the pair of relations
\begin{multline}
\label{e:pent1intro}
P_f(X_0,X_1,X_2)+\rho^{k-1}P_f(X_0,-\rho^2 X_0+\rho^2 X_1, \rho^2 X_0+X_1+X_2)
\\
=
(-\rho^2)^{k-1}P_f(X_0,\rho X_0 - \rho X_1, \rho^2 X_0+X_1+X_2)
\end{multline}
and
\begin{multline}
\label{e:pent2intro}
P_f(X_0,X_1,X_2)+\rho^{2k-2}P_f(X_0, X_0+\rho X_1, \rho X_0-\rho X_1+X_2)\\
=
(-\rho)^{k-1}P_f(X_0,-\rho X_0 - \rho^2 X_1, \rho X_0+-\rho X_1+X_2)
\end{multline}
coming from a five factor ``pentagonal'' defining relation on $PU(2,1;\mathbb{Z}[\rho]).$

Moreover, every defining relation in a certain presentation for $PU(2,1;\mathbb{Z}[\rho])$ yields one or more of the above relations on $P_f(X_0,X_1,X_2).$ 
\end{introthmB}

The relation (\ref{e:R^2intro}) is the Picard period polynomial analogue of (\ref{e:DA1intro}) while 
(\ref{e:PMintro}) corresponds to the Manin relation (\ref{e:DA2intro}).

Owing to the fact that $\mathcal{G}_{(01,01)}^{(4)}$ turns out to be a free group, an explicit analogue of the topological part of Theorem A has no content. Nevertheless, if $\mathfrak{z}$ denotes a tangential basepoint of $\mathcal{M}_{0,5}$,  as a consequence of Theorem B we do have the 
\begin{introthmC}
Every contractible path in $\pi_1(\mathcal{M}_{0,5}, \mathfrak{z})$ gives rise to a relation on the period polynomial $P_f(X_0,X_1,X_2)$ for the Picard modular form $f$ of Picard type $M$.
\end{introthmC}
The sense in which the relation on $P_f(X_0,X_1,X_2)$ ``arises'' from a given contractible path is as follows: either the monodromy representation can be used to decompose the image of the path in $PU(2, 1;\mathbb{Z}[\rho])$ in terms of generators of the latter group, or the relation on $P_f(X_0,X_1,X_2)$ is trivial. 
 The difference from Theorem B is that the $S_4$ action associated to $PU(2,1;\mathbb{Z}[\rho])$ is not faithful on the set of relevant tangential basepoints of $\mathcal{M}_{0,5}$. This is why the pentagonal path (cf. \cite{Ihara:braids}) does not decompose into a five term relation on the Picard period polynomials. 


\section*{Acknowledgements}
I would like to thank Masanobu Kaneko for a very helpful comment which steered me towards a way to complete the proof of the main result. I would also like to thank Minhyong Kim for his sustained interest in my work. 




\section{The elliptic curve case}
\label{s:1}
\subsection{From elliptic curves to paths in $\mathcal{M}_{0,4}$ corresponding to elements of $PSL(2, \mathbb{Z})$}
\label{s:1.1}

The Legendre form of the Weierstrass equation of an elliptic curve $E_{\lambda}$ is 
\[
y^2 = x(x-1)(x-\lambda)
\]
and for each $\lambda \in \mathbb{P}^1\backslash\{0,1,\infty\}$ the curve $E_{\lambda}$ is isomorphic to the curves $E_{f(\lambda)}$ where $f(\lambda)$ is an element of the anharmonic group of transformations of $\mathbb{P}^1\backslash\{0,1,\infty\}$ given by 
\begin{equation}
\label{e:anharmonic}
\Lambda_{\lambda} := \left\{\lambda, 1-\lambda, \frac{1}{1-\lambda}, 
          \frac{\lambda}{\lambda - 1}, \frac{1}{\lambda}, \frac{\lambda - 1}{\lambda} \right\}.
          \end{equation}
There are six such curves for a given choice of $\lambda$, unless two or more of the elements of $\Lambda_{\lambda}$ are equal, which occurs only when $\lambda$ is a primitive 6th root of unity or equals $-1, 2$ or $\frac{1}{2}$ (cf. Proposition 1.7 in \cite{SilvermanAEC}). 

Fix $\lambda \in \mathbb{P}^1\backslash\{0,1,\infty\}$. Then we can regard $E_{\lambda}$ as a two-sheeted cover of the $x$-sphere by making branch cuts on two respective copies $P_0$ and $P_1$ of $\mathbb{P}^1\backslash\{0,1,\infty\}$ along the geodesics $\gamma^j_{0}$ between 0 and $\infty$ and $\gamma^j_{1}$ between $1$ and $\lambda$  in $P_j$ for $j=0,1$, and gluing along these cuts. 
Now let $A_{\lambda}$ denote a loop in $E_{\lambda}$ about 0 and 1 which crosses each of the four geodesics $\gamma^j_k$ exactly once, and let $B_{\lambda}$ denote a loop in $P_0$ about 1 and $\lambda$ which does not cross any of the four geodesics $\gamma_j^k$. Then $\{A_{\lambda},B_{\lambda}\}$ may be regarded as a basis of the homology group $H_1(E_{\lambda}, \mathbb{Z}).$

The periods of $E_{\lambda}$ are then
\[
\omega_1 = \int_{A_{\lambda}} \frac{dx}{y}
\;\;\;\;\;\;\; {\text{ and }}\;\;\;\;\;\;\;
\omega_2 = \int_{B_{\lambda}}\frac{dx}{y},
\]
and $\tau_{\lambda}:= \omega_{2}/\omega_1$ lies in  $\mathfrak{H},$ the upper half plane.

If we now let $\lambda$ vary in $\mathbb{P}^1\backslash\{0,1,\infty\}$, then 
$
\lambda \mapsto \tau_{\lambda}
$
is a multi-valued map $\Psi$ of 
$
\mathbb{P}^1\backslash\{0,1,\infty\}\rightarrow
\mathfrak{H},
$ the inverse of which is the classical lambda function giving the uniformization of $\mathbb{P}^1\backslash\{0,1,\infty\}$ by $\mathfrak{H}$. 

Given $\gamma \in \pi_1(\mathbb{P}^1\backslash \{0,1,\infty\}, \lambda_0)$ for some basepoint $\lambda_0,$
 the analytic continuation of the restriction of $\Psi$ to some simply connected neighborhood of $\lambda_0$  along $\gamma$ gives rise to a map $\gamma \mapsto N(\gamma)$ via the transformation of $\Psi$ to $N(\gamma) \circ \Psi$ where $N(\gamma)$ is some automorphism of $\mathfrak{H}$ - i.e. $N(\gamma) \in PSL(2, \mathbb{R}).$ In other words, we have a map $\pi_1(\mathbb{P}^1\backslash\{0,1,\infty\} , \lambda_0)\rightarrow PSL(2, \mathbb{R})$, the image of which is called the monodromy group of $\Psi.$  

Denote loops based at $\lambda_0$ about 0 and 1 respectively by  $\sigma_0$ and $\sigma_1.$ These can be viewed as the generators of $\pi_1(\mathbb{P}^1\backslash\{0,1,\infty\} , \lambda_0)$. Now $N(\sigma_j)$ can be computed via the effect of the analytic continuation along $\sigma_j$ on the homology basis $\{A_{\lambda}, B_{\lambda}\}$. One sees that 
\[
N(\sigma_0)\; ^t[B_{\lambda}\;\; A_{\lambda}]
=
\left[ 
\begin{array}{cc}
1 & 2
\\
0 & 1 
\end{array}
\right] 
\left[
\begin{array}{c}
B_{\lambda}
\\
A_{\lambda}
\end{array}
\right]
\]
while 
\[
N(\sigma_1)\; ^t[B_{\lambda}\;\; A_{\lambda}]
=
\left[ 
\begin{array}{cc}
-1 & 0
\\
2 & -1 
\end{array}
\right] 
\left[
\begin{array}{c}
B_{\lambda}
\\
A_{\lambda}
\end{array}
\right].
\]
Denoting generators of $PSL(2, \mathbb{Z})$ by 
\begin{equation}
\label{e:ST}
S = \left[\begin{array}{cc}
0 & 1
\\ -1 & 0
\end{array}
\right]
\;\;\;\;\;\; {\text{ 
and }}\;\;\;\;\;\;
T = \left[ \begin{array}{cc}
1 & 1
\\
0 & 1 
\end{array}
\right]
\end{equation}
one finds that $N(\sigma_0) = T^2$ while $N(\sigma_1) = ST^2S$. Since the latter matrices generate the congruence subgroup $\Gamma(2) \leq PSL(2, \mathbb{Z})$, which is a free group on the two generators, in this case the monodromy representation is faithful.

Now as Schneps explains in \S 2.5 of \cite{Buffetal:99}, the fundamental group of the moduli space $\mathcal{M}_{0,n}$ of genus 0 curves with $n$ marked points may be identified with the kernel $K(0,n)$ of the projection of the mapping class group $M(0,n)$ to $S_n$ (where $n$ should be replaced by $n-1$ when $n=4$);
and this identification may be extended to a mapping from $M(0,n)$ itself to a certain set of path classes emanating from a so-called real point of $\mathcal{M}_{0,n}$. When the real point in question admits no automorphisms (which could arise from the orbifold structure on $\mathcal{M}_{0,n}$) the map is an isomorphism and the path space thus acquires a group structure. 

Schneps gives her map for the configuration space model of $\mathcal{M}_{0,n}$ but an explicit version of this map in the complex model would be very useful for computations. In \cite{Joyner:qm} the author gave such a map for the $n=4$ case. This yields a correspondence between the elements of all of $M(0,3) = PSL(2, \mathbb{Z})$, and certain homotopy classes of paths in $\mathbb{P}^1\backslash\{0,1,\infty\}$. In order to describe this correspondence, first we need to introduce the notion of tangential basepoint, which is due to Deligne \cite{Deligne:89}. To this end, suppose that $X = \overline{X}\backslash S$ is a smooth curve over $\mathbb{C}$ where $S$ denotes some finite set of points. Then at any omitted point $a\in S$, it is possible to define the fundamental group of $X$ in the direction of a specified tangent vector to $\overline{X}$ at $a$. A classical way to do this (cf. for example \cite{Hain:CP}) is to set
\[
P_{{v}_0, {v}_1}:=
\{\gamma:[0,1] \rightarrow \overline{X} \;|\; 
\gamma'(0) = {v}_0; \gamma'(1) = -{v}_1; \gamma((0,1)) \subset X\}
\]
(where ${v}_j \in T_{a_j}$ is a tangent vector at $a_j \in S$ for $j=0,1$), and then denote the set of path components of $P_{{v}_0,{v}_1}$ by $\pi_1(X, {{v}_0, {v}_1})$. When ${v}_0 = {v}_1$ this gives the fundamental group based at ${v}_0$, denoted $\pi_1(X, {v}_0)$.

Now the anharmonic group $\Lambda_{\lambda} \simeq S_3$ of (\ref{e:anharmonic}) both permutes $0,1$ and $\infty$, and gives the group of all linear fractional transformations of $\mathbb{P}^1\backslash\{0,1,\infty\}$. Therefore given a homotopy class 
$
\alpha \in \pi_1(\mathbb{P}^1\backslash \{0,1,\infty\}, \overrightarrow{01}, \overrightarrow{ab})
$ 
where $a,b$ are distinct elements of $\{0,1,\infty\}$ and $\overrightarrow{ab}$ denotes the unit vector based at $a$ in the direction of $b$ along the real line in $\mathbb{P}^1(\mathbb{C}),$ the permutation $\sigma_{ab} \in S_3$ sending $0$ to $a$ and $1$ to $b$ corresponds to a unique element of $\Lambda_{\lambda}$, say $\lambda_{ab}$, so given any $
\beta \in \pi_1(\mathbb{P}^1\backslash \{0,1,\infty\}, \overrightarrow{01}, \overrightarrow{cd})
$ 
for any distinct $c,d \in \{0,1,\infty\}$ then $
\lambda_{ab}(\beta) \in \pi_1(\mathbb{P}^1\backslash \{0,1,\infty\}, \overrightarrow{ab}, \overrightarrow{\sigma_{ab}c\;\sigma_{ab}d}).$ 
But then we can define the product of $\alpha$ and $\beta$ in 
\[
\mathcal{G}_{01}\;:= \bigcup_{a,b \in \{0,1,\infty\}, a \neq b} \pi_{1}(\mathbb{P}^1\backslash\{0,1,\infty\}, \overrightarrow{01}, \overrightarrow{ab})
\]
as the path $\alpha$ followed by $\lambda_{ab}(\beta).$ Now let $s$ denote the tangential path along the unit interval from $\overrightarrow{01}$ to $\overrightarrow{10}$, and let $t$ denote the tangential path from $\overrightarrow{01}$ to $\overrightarrow{0\infty}$ which is a loop tangential to the real line lying in the upper half plane. In \cite{Joyner:qm}, $\mathcal{G}_{01}$ with the group operation described above is shown to be isomorphic to $PSL(2, \mathbb{Z})$ under the association of $s$ and $t$ respectively with the matrices $S$ and $T$ of (\ref{e:ST}). With this notation a presentation for $\mathcal{G}_{01}$ is:
\[
\mathcal{G}_{01} = \;< s, t \;|\;s^2 = (st)^3 = 1 >
.
\]
The short exact sequence 
\begin{equation}
\label{e:PSLses}
1 \rightarrow K(0,3) \rightarrow PSL(2, \mathbb{Z}) \rightarrow S_3 \rightarrow 1
\end{equation}
is then faithfully represented on 
\begin{equation}
\label{e:P1ses}
1 \rightarrow \pi_1(\mathbb{P}^1\backslash\{0,1,\infty\}, \overrightarrow{01}) \rightarrow \mathcal{G}_{01} \rightarrow S_3 \rightarrow 1,
\end{equation}
where the multiplication on $\mathcal{G}_{01}$ is ``twisted'' by the action of $S_3$ on $\mathbb{P}^1\backslash\{0,1,\infty\}$ as described above.

\subsection{The period polynomials}
\label{s:1.2}

A function $f$ of $\mathfrak{H}$ into $\mathbb{C}$ is a cusp form for $PSL(2, \mathbb{Z})$ if $f$ admits of a Fourier series expansion of the form 
\begin{equation}
\label{e:Fourier}
f(q) = \sum_{n=1}^{\infty}a_nq^n
\end{equation}
where $q = e^{2\pi i \tau}$ for $\tau \in \mathfrak{H}$, and satisfies 
\[
f(\gamma (\tau)) = j_{\gamma}^k(\tau) f(\tau)
\]
for every $\tau \in PSL(2, \mathbb{Z})$, where $j_{\gamma}(\tau)$ is the Jacobian determinant $(a_{10}\tau + a_{11})$ if $\gamma = (a_{ij})_{0\leq i,j \leq 1}$. Then
\[
\omega_f: = f(\tau) (X_0 \tau+ X_1)^k d \tau
\]
is a differential form on the modular curve since it is invariant under the action of $PSL(2, \mathbb{Z})$ defined by the usual action 
\[
\tau \mapsto \frac{a_{00} \tau +a_{01}}{a_{10}\tau +a_{11}}
\]
of $\gamma = (a_{ij})_{0\leq i,j \leq 1}$ on $\mathfrak{H}$ and the analogous action induced by $^t\gamma^{-1}$ on $(X_0,X_1)$ - i.e. for $k=0,1$
\begin{equation}
\label{e:fva}
X_k \mapsto (-1)^{k}a_{(1-k)1}X_0 + (-1)^{k+1}a_{(1-k)0}X_1.
\end{equation}
The latter is not how the action on formal variables - which ought to be thought of as basis elements for homology (cf. \cite{BrownHainDeRham}) - is usually defined, but will lend itself to generalization later. 

Then we define
the (elliptic) period polynomial associated to the cusp form $f$ to be 
\[
P_f(X_0,X_1):= \int_{{i \infty}\rightarrow 0} f(\tau) (X_0 \tau+ X_1)^k d \tau
\]
where $i\infty \rightarrow 0$ denotes the (geodesic) path from $i \infty$ to $0$ along the imaginary axis in $\mathfrak{H}$. 
The integral converges because of (\ref{e:Fourier}) - cf. \cite{LangIntroModForms}. 

The classical lambda function $\lambda$ which gives the uniformization of $\mathbb{P}^1\backslash\{0,1,\infty\}$ by $\mathfrak{H}$ maps the path of integration to the tangential path $s$ (cf. \cite{Chandra}), which we associate with the element $S$ of $PSL(2, \mathbb{Z})$. 

\begin{theorem}
\label{t:1}
Every contractible path in $\pi_1(\mathbb{P}^1\backslash\{0,1,\infty\}, \overrightarrow{01})$ gives a relation on $P_f(X_0,X_1)$. In particular, 
\begin{equation}
\label{e:DA1}
P_f(X_0,X_1) = -P_f(-X_1, X_0)
\end{equation}
corresponds to the path $s \circ s = 1$ and 
\begin{equation}
\label{e:DA2}
P_f(X_0,X_1) + P_f(-X_0-X_1, X_0) + P_f(X_1, -X_0-X_1) = 0
\end{equation}
corresponds to the contractibility of the hexagonal path $(s\circ t)^3$.
\end{theorem}

{\bf Remark:} The equation (\ref{e:DA1}) is the linearized (Lie algebra) version of the equation
\[
\Phi_{KZ}(A, B) = \Phi_{KZ}(-B,-A)^{-1}
\]
satisfied by the Drinfel'd associator $\Phi_{KZ}(A,B)$ (cf. \cite{Cartier}). Here the action of $S$ on the formal variables $A$ and $B$ comes from the action of the anharmonic group on the Knizhnik-Zamolodchikov (KZ) equation (cf. \cite{Joyner:qm}).

Similarly, (\ref{e:DA2}) is the logarithmic analogue of the hexagonal relation
\[
\Phi_{KZ}(A,B)\exp(i \pi B) \Phi_{KZ}(-B, A-B)\exp(i \pi (A-B)) \Phi_{KZ}(B-A,-A) \exp(i\pi (-A)) = 1
\]
In \cite{Joyner:qm} it is shown that these equations for $\Phi_{KZ}$ may be seen to arise from analytic continuation of a specific flat section (namely the polylogarithm generating function) of the universal prounipotent bundle with connection on $\mathbb{P}^1\backslash\{0,1,\infty\}$ along $s\circ s$ and $(s\circ t)^3$ respectively.

{\em Proof of Theorem \ref{t:1}.}  
Since the fundamental group $\pi_1(\mathbb{P}^1\backslash\{0,1,\infty\}, \overrightarrow{01})$ is generated by $t^2$ and $s\circ t^2 \circ s$, any contractible loop $\gamma$ in $\mathbb{P}^1\backslash \{0,1, \infty\}$ based at $\overrightarrow{01}$, regarded as an element of this fundamental group, can be decomposed as the concatenation of the paths $s$ and $t$; and by prepending the trivial loop $s\circ s$ if necessary, $\gamma$ is homotopic to a loop $\gamma_s$ beginning with the path $s$. Then $\gamma_s$ lifts to a contractible path in a fundamental domain for $\Gamma(2)$ in the $PSL(2, \mathbb{Z})$ action on $\mathfrak{H}$, for example (again cf. \cite{Chandra}) the set 
\[
\mathcal{F}:=\{\tau \in \mathfrak{H}\;|\;
-1<\Re(\tau)\leq 1; \;
\left|\tau-\frac{1}{2}\right|\geq\frac{1}{2};\;\left|\tau+\frac{1}{2}\right|> \frac{1}{2}\},
\]
which starts with the path $i\infty \rightarrow 0$. Because of the one-to-one correspondence between $\mathcal{G}_{01}$ and $PSL(2, \mathbb{Z})$, this lifted path (say $G_s$) can be viewed as the union of $i\infty \rightarrow 0$ under the action of a succession of elements of $PSL(2, \mathbb{Z})$. Since $G_s$ is contractible, the composition of these elements of $PSL(2, \mathbb{Z})$ gives the identity and thus we have a corresponding relation in $PSL(2, \mathbb{Z})$,
 say $A_n \circ A_{n-1}\circ  \cdots\circ A_1\circ A_0 = I$. Consequently we have the equation:
\begin{align*}
0 &= \int_{G_s} f(\tau)(X_0\tau+X_1)^k d \tau
\\
&=
\int_{A_0(i\infty\rightarrow 0)} + \int_{A_1 \circ A_0 (i \infty\rightarrow 0)} + \cdots + \int_{A_{n-1}\circ \cdots \circ A_0(i \infty\rightarrow 0)} + \int_{i \infty \rightarrow 0} f(\tau) (X_0\tau + X_1)^k d \tau
\end{align*}

Now if $A \in PSL(2, \mathbb{Z})$, notice that
\begin{align*}
& \int_{A(i \infty \rightarrow 0)} f(\tau) (X_0\tau + X_1)^k d \tau
\\
&=
\int_{i \infty\rightarrow 0} f(A^{-1}\tau) (X_0A^{-1}\tau + X_1)^k d A^{-1}\tau
.
\end{align*}
But here, $f(\tau) (X_0\tau + X_1)^k d \tau$ is invariant under the $A^{-1}$ action. Hence, 
\[
f(A^{-1}\tau) (X_0 A^{-1}\tau + X_1)^k d A^{-1}\tau
=
f(\tau)(AX_0 \tau + AX_1)^k d\tau
\]
where $AX_j$ denotes the image of $X_j$ under the action of $A$ given by (\ref{e:fva}) above for $j=0,1$.

The statements pertaining to the trivial paths $s\circ s$ and $(s\circ t)^3$ respectively now follow from straightforward calculations, although a few words are in order to explain the vanishing of certain terms in the second relation.

First we should remark that since the cusps $i\infty$ and 0 are not in $\mathfrak{H}$ the relevant paths must be suitably interpreted. This is easily handled using a limiting procedure analogous to how we are thinking about tangential basepoints in $\mathbb{P}^1\backslash\{0,1,\infty\}$.

Now we lift the trivial path around the real line in $\mathbb{P}^1\backslash\{0,1,\infty\}$ which avoids the punctures 1, $\infty$ and 0 via images of $t^{-1}$ based at $\overrightarrow{10},$ $\overrightarrow{\infty1}$ and $\overrightarrow{0\infty}$ respectively - i.e. we lift the path $(s\circ t^{-1})^3=1.$ Since $s \circ s = 1$ this path is equivalent to $(s\circ t)^3 = 1$. One checks (cf. \cite{Chandra}) that the upper half of $\mathbb{P}^1\backslash\{0,1,\infty\}$ is covered by the part of $\mathcal{F}$ with positive real part under the  uniformizing function $\lambda$. The lifting of $(s\circ t^{-1})^3 = 1$ is the boundary of this half of $\mathcal{F}$, with the cusps omitted due to the lifts of $t^{-1}$. To be more precise, notice that in $\mathcal{F}$, the geodesic from 0 to $i\infty$ and the geodesic from 0 to 1 are tangential at 0. Therefore the lift of the image of $t^{-1}$ under the element of the anharmonic group $z \mapsto 1-z$ corresponding to $s,$ to $\mathcal{F},$ is a path from the tangential basepoint $\overrightarrow{0, i \infty}$ to $\overrightarrow{0, 1}$ (which are identical elements of the tangent space to 0), say $\varsigma_0$, which effectively omits the cusp from the region bounded by the lift of $(s\circ t^{-1})^3$. The other two lifts of $t^{-1}$, say $\varsigma_1$ and $\varsigma_{i\infty}$, similarly omit the cusps at 1 and $1+i \infty = i \infty$. (For the latter this is clear from the disc model of $\mathfrak{H}.$) 

But then the integrals of $f(\tau)(X_0\tau+X_1)^k d\tau$ along $\varsigma_j$ are zero for each of $j=0,1$ and $i\infty.$ $\hfill \Box$

\section{The Picard curve case}
\label{s:2}
\subsection{From Picard curves to a representation of a certain group of paths in $\mathcal{M}_{0,5}$ on $PU(2,1; \mathbb{Z}[\rho])$}
\label{s:2.1}

Ideally we would like to emulate the $\mathcal{M}_{0,4}$ situation of the faithful representation of path classes on the modular group, in the case of $\mathcal{M}_{0,5}$. However, the monodromy representation onto a certain linear group $\Gamma_1$, which arises through a similar process to that discussed above in \S \ref{s:1.1}, turns out not to be faithful. Nevertheless, a monodromy representation restricted to a certain subgroup of the fundamental group onto $\Gamma_1$ can again be extended to a representation of a certain space of paths emanating from some (tangential) basepoint in $\mathcal{M}_{0,5},$ onto a suitable analogue of $PSL(2, \mathbb{Z})$, namely $PU(2, 1; \mathbb{Z}[\rho])$, and this section is devoted to developing this representation explicitly.

The curves $C(x_0,y_0)$ given by the equation
\[
w^3 = z(z-1)(z-x_0)(z-y_0)
\]
for $[x_0: y_0:1]$ in the space 
\[
M:= \{[x_0:y_0:1] \in \mathbb{P}^2(\mathbb{C})\;: \; x_0y_0(x_0-1)(y_0-1)(x_0-y_0) \neq 0\}
\]
are genus 3 compact curves known as Picard curves. 
One 
sees immediately that the parameter space is none other than the moduli space $\mathcal{M}_{0,5}$ of genus 0 curves with 5 marked points since this space may be seen to be isomorphic to
\[
(\mathbb{P}^1\backslash\{0,1,\infty\})^2 \backslash\Delta 
\]
where $\Delta =
\{(z,z)\in \mathbb{P}^2(\mathbb{C})\;:\; z \in \mathbb{P}^1\backslash\{0,1,\infty\}\},
$ by using the cross ratio. 
Now, as in the elliptic curve case discussed in \S \ref{s:1.1} above, analytic continuation along paths in the parameter space $\mathcal{M}_{0,5}$ gives rise to an action on periods of the Picard curves (which could also be viewed as an action on a homology basis for $H_1(C(x_0,y_0), \mathbb{Z})$ for some fixed $(x_0,y_0)\in \mathcal{M}_{0,5}$). This is discussed by Shiga in \cite{Shiga} so we omit the details but summarize the results we will use here.

Pick $x_0$ and $y_0$ on the positive real axis with $1< x_0 < y_0.$ Then, following the notation of \cite{Shiga}, set
\[
L_x = \{(x,y) \in \mathbb{C}^2 \;:\; y = y_0\}
\]
and 
\[
L_y:= \{(x,y) \in \mathbb{C}^2 \;:\;x = x_0\}
\]
and define the following positively oriented loops in $L_x$ based at $\alpha_0 = (x_0,y_0)$:
\\
$\delta_1$, a loop about $x=1$;
\\
$\delta_{y}$, a loop about $x=y_0$; 
\\
$\delta_0$, a loop about $x=0$; 
\\
and the positively oriented loops in $L_y$ also based at $\alpha_0$: 
\\
$\delta_{y,0}$, a loop about $y=0$; and 
\\
$\delta_{y,\infty}$, a loop about $y=\infty$. 
\\
These are to be regarded as paths in $\pi_1(\mathcal{M}_{0,5}, \alpha_0)$ along which we deform $C(x_0,y_0)$ by allowing the parameter $(x_0,y_0)$ to vary. This induces a monodromy transformation as in the elliptic curve case. Here analytic continuation along a path $\delta$ transforms the period mapping by an element $N_5(\delta)$ of $Sp(3, \mathbb{Z})$ (which acts upon the Siegel upper half space $\mathfrak{H}_3$) and Shiga works out the explicit symplectic transformations corresponding to the $\delta_j$ given above. In this case, there exist relations among three respective pairs of the periods
\[
\int_{E}\frac{dz}{w}
\]
for elements $E$ of a homology basis for $H_1(C(x_0,y_0), \mathbb{Z})$ and consequently the action of each symplectic element $N_5(\delta)$ induces an action of an element (say $g(\delta)$) of $PGL(3, \mathbb{Z}[\rho])$ on a triple of consisting of one period for each relation. The $g(\delta)$ act upon the two dimensional complex ball $\mathbb{H}^2_{\mathbb{C}}$ (hyperbolic space), which is given in non-homogeneous coordinates by 
\begin{equation}
\label{e:ball}
\{[z_1 : z_2 : 1] \in \mathbb{P}^2(\mathbb{C}) \; : \;
2 \text{Re}(z_1)+|z_2|^2 <0\}.
\end{equation}
In keeping with the notation of \cite{FalbelParkerPicard} which it will be convenient to use later, the matrices $g(\delta_j)$ computed by Shiga are given here after conjugation by 
\[
W:=
\left[
\begin{array}{ccc}
0&0&1
\\
1&0&0
\\
0&1&0
\end{array}\right].
\]
To wit, they are:
\begin{align*}
g(\delta_1)
 &= 
 \left[
\begin{array}{ccc}
1&0&0
\\
0&\rho&0
\\
0&0&1 
\end{array}\right]
\qquad &
g(\delta_{y,0})
& = 
 \left[
\begin{array}{ccc}
\rho^2&0&\rho - 1
\\
0&1&0
\\
\rho - 1 &0&-2\rho^2
\end{array}\right]
\\
g(\delta_y)&
 = 
 \left[
\begin{array}{ccc}
1&1- \rho^2&\rho -1
\\
0&\rho&\rho^2 - \rho
\\
0&0&1
\end{array}\right]
&g(\delta_{y, \infty})
 &= 
 \left[
\begin{array}{ccc}
1&0&0
\\
0&1&0
\\
\rho - \rho^2&0&1
\end{array}\right]
\end{align*}
and
\[
g(\delta_0) = 
 \left[
\begin{array}{ccc}
1&0&0
\\
\rho-1&\rho&0
\\
\rho - 1&\rho -1 &1
\end{array}\right].
\]
These matrices generate the monodromy representation
\[
\Gamma_1 :=\{g(\delta) \in PGL(3, \mathbb{Z}[\rho])\;:\;\delta \in \pi_1(\mathcal{M}_{0,5}, \alpha)\}.
\]

Now consider the matrix 
\[
J_0:=
\left[
\begin{array}{ccc}
0&0&1
\\
0&1&0
\\
1&0&0
\end{array}
\right]
\]
and set 
\begin{equation}
\label{e:FPPconv}
\Gamma=
\{g \in PGL(3, \mathbb{Z}[\rho])\;:\;
^tgJ_0\bar{g} = J_0\}
\end{equation}
Then $\Gamma = PU(2,1;\mathbb{Z}[\rho])$ consists of the elements of the projectivization of the unitary group preserving the Hermitian form $<r,s>:=s^*J_0r$ on $\mathbb{C}^3$, which can be represented by matrices with entries in $\mathbb{Z}[\rho]$. 

Here the analogue of (\ref{e:PSLses}) is 
\begin{equation}
\label{e:PUses}
1\rightarrow \Gamma_1 \rightarrow
PU(2,1;\mathbb{Z}[\rho])\rightarrow S_4 \rightarrow 1.
\end{equation}

Now in \cite{Shiga}, an action of $S_4$ on $(\mathbb{P}^1\backslash\{0,1,\infty\})^2\backslash \Delta$ is given via explicit matrices, induced by the permutation of the points 
\[
\mathfrak{a}_1=[0:0:1], \;
\mathfrak{a}_2 = [0:1:0], 
\;
\mathfrak{a}_3 = [1:0:0] \;
\text{ and }
\;
\mathfrak{a}_4 = [1:1:1]
\]
of $\mathbb{P}^2.$

We remark that the mapping class group $M(0,5)$ (the quotient of the Artin braid group by the sphere and center relations into which the fundamental group of $\mathcal{M}_{0,5}$ embeds - see \S 2.4 in \cite{Buffetal:99}  for the details) coincides with  $\tilde\pi_1\left((\left[ \mathcal{M}_{0,5} /S_5\right], *\right)$, where $\left[ \mathcal{M}_{0,5} /S_5\right]$ denotes the orbifold quotient and $\tilde\pi_1$ is the  orbifold fundamental group. By the same token, an injection of the fundamental group of $\mathcal{M}_{0,5}$ into $\tilde\pi_1\left((\left[ \mathcal{M}_{0,5} /S_4\right], *\right)$ exists. The monodromy representation thereof (say $\mathcal{Q}$) is onto $PU(2,1; \mathbb{Z}[\rho])$ (cf. \cite{Shiga}). Hence we have:

\begin{tikzpicture}
  \matrix (m)
    [
      matrix of math nodes,
      row sep    = 1.5em,
      column sep = 1.5em
    ]
    {
      1              
      & \pi_1((\mathbb{P}^{1}\backslash\{0,1,\infty\})^2\backslash \Delta, (\overrightarrow{01}, \overrightarrow{01})) 
      &
      & \tilde\pi_1\left((\left[ \mathcal{M}_{0,5} /S_4\right], *\right)
      & S_4
      & 1 \\
     1              
      & \Gamma_1      && PU(2,1;\mathbb{Z}[\rho])
      & S_4
      & 1 \\ 
    };
  \path
  (m-1-1) edge [->] (m-1-2) 
  (m-1-2) edge[->] (m-1-4) 
  (m-1-5) edge[->](m-1-6) 
  (m-1-4) edge [->] (m-1-5)
    (m-1-2) edge [->>] node [left] {$\mathcal{P}$} (m-2-2)
    (m-1-4) edge [->>] node [left] {$\mathcal{Q}$} (m-2-4)
          (m-2-1) edge [->] (m-2-2) 
  (m-2-2) edge[->] (m-2-4) 
  (m-2-5) edge[->](m-2-6) 
  (m-2-4) edge [->] (m-2-5)      ;
\end{tikzpicture}

However, our goal is to parallel the $\mathcal{M}_{0,4}$ case as closely as possible so we would like to give a representation of a group of paths on $\mathcal{M}_{0,5}$ {\em itself} onto $PU(2,1; \mathbb{Z}[\rho])$. This will take more work. 

A presentation for $S_4$ is
\[
<s_1,s_2,s_3\;|\; s_1^2 = s_2^2 = s_3^2 = (s_1s_2)^3 = (s_1s_3)^2 = (s_2s_3)^3 = e>
\]
so once and for all set $s_1 = (12),$ $s_2 = (24)$ and $s_3 = (34)$.


 
The $S_4$ action on $\mathcal{M}_{0,5}$ induces an action on the set, say $\mathcal{B}$, of tangential basepoints for $(\mathbb{P}^1\backslash\{0,1,\infty\})^2\backslash \Delta$ at the points $\mathfrak{a}_j,$ $1\leq j \leq 4$, ``at infinity''. At each $\mathfrak{a}_j$ lie four of the sixty tangential basepoints of $\mathcal{M}_{0,5}$ (cf. \cite{Buffetal:99}, \S 1.3.3 (ii)). Consequently some choices are required to make the $S_4$ action on $\mathcal{B}$ explicit. The version of this action we give here 
was chosen in such a way that the action of the transposition $(14)$ reverses the direction of the tangential arrows, since under Shiga's $S_4$ action this transposition induces the automorphism $(x,y) \mapsto (1-x,1-y)$ of $(\mathbb{P}^1\backslash\{0,1,\infty\})^2\backslash\Delta$. 

The 16 tangential basepoints of $\mathcal{B}$ are permuted under the $S_4$ action generated by the following automorphisms (recall that we are fixing real $x_0$ and $y_0$ in $\mathbb{P}^1\backslash\{0,1,\infty\}$ with $1<x_0<y_0$):
\clearpage

\begin{tabular}{|l|c|c|c|}
\hline
Basepoint & $(12)$ & $(24)$ &$(34)$ 
\\
\hline
$(\overrightarrow{01}, \overrightarrow{01})$ & 
$(\overrightarrow{1x_0}, \overrightarrow{\infty y_0})$ & 
$(\overrightarrow{01}, \overrightarrow{0\infty})$ & 
$(\overrightarrow{01}, \overrightarrow{0\infty})$ 
\\
$(\overrightarrow{01}, \overrightarrow{0\infty})$ & 
$(\overrightarrow{10}, \overrightarrow{\infty 0})$ & 
$(\overrightarrow{01}, \overrightarrow{01})$ & 
$(\overrightarrow{01}, \overrightarrow{01})$ 
\\
$(\overrightarrow{0\infty}, \overrightarrow{01})$ & 
$(\overrightarrow{10}, \overrightarrow{\infty y_0})$ & 
$(\overrightarrow{0\infty}, \overrightarrow{0\infty})$& 
$(\overrightarrow{0\infty}, \overrightarrow{0\infty})$  
\\
$(\overrightarrow{0\infty}, \overrightarrow{0\infty})$ & 
$(\overrightarrow{1x_0}, \overrightarrow{\infty 0})$ & 
$(\overrightarrow{0\infty }, \overrightarrow{01})$ & 
$(\overrightarrow{0\infty}, \overrightarrow{01})$ 
\\
$(\overrightarrow{10}, \overrightarrow{10})$ & 
$(\overrightarrow{10}, \overrightarrow{1x_0})$ & 
$(\overrightarrow{10}, \overrightarrow{\infty 0})$ & 
$(\overrightarrow{\infty 0}, \overrightarrow{10})$ 
\\
$(\overrightarrow{10}, \overrightarrow{1x_0})$ & 
$(\overrightarrow{10}, \overrightarrow{10})$ & 
$(\overrightarrow{1x_0}, \overrightarrow{\infty y_0})$ & 
$(\overrightarrow{\infty y_0}, \overrightarrow{1x_0})$ 
\\
$(\overrightarrow{1x_0}, \overrightarrow{1x_0})$ & 
$(\overrightarrow{1x_0}, \overrightarrow{10})$ & 
$(\overrightarrow{10}, \overrightarrow{\infty y_0})$ & 
$(\overrightarrow{\infty y_0}, \overrightarrow{10})$ 
\\
$(\overrightarrow{1x_0}, \overrightarrow{10})$ & 
$(\overrightarrow{1x_0}, \overrightarrow{1x_0})$ & 
$(\overrightarrow{1x_0}, \overrightarrow{\infty 0})$ & 
$(\overrightarrow{\infty 0}, \overrightarrow{1x_0})$ 
\\
$(\overrightarrow{10}, \overrightarrow{\infty 0})$ & 
$(\overrightarrow{01}, \overrightarrow{0\infty})$ & 
$(\overrightarrow{10}, \overrightarrow{10})$ & 
$(\overrightarrow{1x_0}, \overrightarrow{\infty y_0})$ 
\\
$(\overrightarrow{10}, \overrightarrow{\infty y_0})$ & 
$(\overrightarrow{0\infty}, \overrightarrow{01})$ & 
$(\overrightarrow{1x_0}, \overrightarrow{1x_0})$ & 
$(\overrightarrow{1x_0}, \overrightarrow{\infty 0})$ 
\\
$(\overrightarrow{1x_0}, \overrightarrow{\infty y_0})$ & 
$(\overrightarrow{01}, \overrightarrow{01})$ & 
$(\overrightarrow{10}, \overrightarrow{1x_0})$ & 
$(\overrightarrow{10}, \overrightarrow{\infty 0})$ 
\\
$(\overrightarrow{1x_0}, \overrightarrow{\infty 0})$ & 
$(\overrightarrow{0\infty}, \overrightarrow{0\infty})$ & 
$(\overrightarrow{1x_0}, \overrightarrow{10})$ & 
$(\overrightarrow{10}, \overrightarrow{\infty y_0})$ 
\\
$(\overrightarrow{\infty 0}, \overrightarrow{10})$ & 
$(\overrightarrow{\infty y_0}, \overrightarrow{1x_0})$ & 
$(\overrightarrow{\infty y_0}, \overrightarrow{1x_0})$ & 
$(\overrightarrow{10}, \overrightarrow{10})$ 
\\
$(\overrightarrow{\infty 0}, \overrightarrow{1x_0})$ & 
$(\overrightarrow{\infty y_0}, \overrightarrow{10})$ & 
$(\overrightarrow{\infty y_0}, \overrightarrow{10})$ & 
$(\overrightarrow{1x_0}, \overrightarrow{10})$ 
\\
$(\overrightarrow{\infty y_0}, \overrightarrow{1x_0})$ & 
$(\overrightarrow{\infty 0}, \overrightarrow{10})$ & 
$(\overrightarrow{\infty 0}, \overrightarrow{10})$ & 
$(\overrightarrow{10}, \overrightarrow{1x_0})$ 
\\
$(\overrightarrow{\infty y_0}, \overrightarrow{10})$ & 
$(\overrightarrow{\infty 0}, \overrightarrow{1x_0})$ & 
$(\overrightarrow{\infty 0}, \overrightarrow{1 x_0})$ & 
$(\overrightarrow{1x_0}, \overrightarrow{1x_0})$ 
\\
\hline
\end{tabular}

Using this action it is now possible to identify a group of homotopy classes of paths in the fundamental groupoid $\Pi_1((\mathbb{P}^1\backslash\{0,1,\infty\} )^2\backslash \Delta; \mathcal{B})$, (which consists of all homotopy classes of paths emanating from any of the tangential basepoints in $\mathcal{B}$ which also end in such a point), which surjects onto $PU(2,1; \mathbb{Z}[\rho])$. 

We begin by introducing useful 
notation: 
reminiscent of the notation $s$ and $t$ in \S \ref{s:1.1}, let  $s_x$ denote any tangential path along the real axis from a given tangential basepoint $(\overrightarrow{ab}, \overrightarrow{cd})$ to  $(\overrightarrow{ba}, \overrightarrow{cd})$, unless $\overrightarrow{ab} = \overrightarrow{\infty y_0}$ in which case $s_x$ runs to $\overrightarrow{x_0 y_0}$. Here $a,b,c,d$ are among $\{0,1,\infty, x_0,y_0\}$ but the first symbol in the $x$-coordinate is never $y_0$ and the first symbol in the $y$-coordinate is never $x_0$; and when the first symbol in the $x$-coordinate is $x_0$ the first symbol in the $y$-coordinate is not $y_0$.   
Similarly denote by $s_y$ the analogue of $s_x$ when the $x$ and $y$ coordinates are switched. Also, let $t_x$ be a loop tangential to the real line, from $(\overrightarrow{ab}, \overrightarrow{cd})$ to  $(\overrightarrow{ae}, \overrightarrow{cd})$ where $b\neq e$, which lies in the upper half plane (in the $x-$coordinate) when $\overrightarrow{ab}$ points in the positive direction along the real axis and $\overrightarrow{ae}$ points in the negative direction; and lies in the lower half plane otherwise. Again let $t_y$ denote the analogous loop when the $x$ and $y$ coordinates are switched.  Of course, each of $s_x, s_y, t_x$ and $t_y$ is an abuse of notation since each represents eight different paths at once, depending on the point at which the path is based. 

Now let $\sigma \in S_4.$ With a similar abuse of notation, let $r_{\sigma}$ denote the path (in fact a family of paths) from any given tangential basepoint in $\mathcal{B}$ to its image basepoint under the action of $\sigma$ determined by the $S_4$ action on $\mathcal{B}$ given above, which consists of the {\em minimal} number of concatenated paths from among $s_x, s_y, t_x$ and $t_y$, unless $r_{\sigma} = (s_x,s_y)$ in which case it should be replaced by $(t_x^2 \circ s_x, s_y)$. For example, viewed as a path emanating from $(\overrightarrow{01},\overrightarrow{01})$, $r_{(12)}$ is the path to $(\overrightarrow{1x_0},\overrightarrow{\infty y_0})$ which in the $x$-coordinate is the concatenation of $s_x$ (from $\overrightarrow{01}$ to $\overrightarrow{10}$) followed by $t_x$ (here in the lower half plane) and in the $y$-coordinate is $t_y$ (in the upper half plane) followed by $s_y$ from $\overrightarrow{0\infty}$ to $\overrightarrow{\infty 0}$ and then $t_y$ (again in the upper half plane). This path is pictured below: 

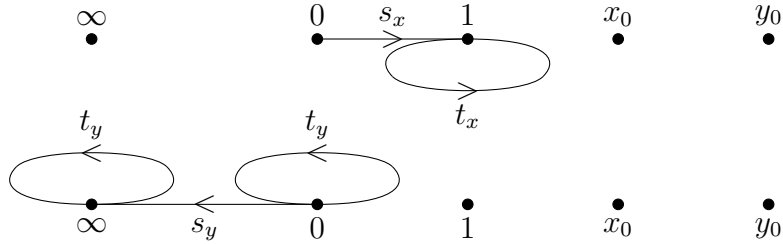
\begin{center}
\begin{tikzpicture}
\filldraw (0,2.2) node [above=1pt] {0} circle (2pt);
\filldraw (2,2.2) node [above=1pt] {1} circle (2pt);
\filldraw (-3,2.2) node[above=1pt] {$\infty$} circle (2pt);
\filldraw (4,2.2) node [above=1pt] {$x_0$} circle (2pt);
\filldraw (6,2.2) node [above=1pt] {$y_0$} circle (2pt);
\draw (0,2.2) --node [above = 1pt] {$s_x$} node {$>$}  (2,2.2);
\draw (2,2.2) .. controls (3, 2.2) and (3.25,1.95) .. (3,1.7);
\draw (3,1.7) .. controls (2.75,1.45) and (1.25,1.45) .. node {$>$} node [below = 1pt] {$t_x$} (1,1.7);
\draw (1,1.7) .. controls (.75,1.95) and (1,2.2) .. (2,2.2);
\draw (0,0) .. controls (1, 0) and (1.25,.25) .. (1,.5);
\draw (1,.5) .. controls (.75,.75) and (-.75,.75) .. node {$<$} node [above = 1pt] {$t_y$} (-1,.5);
\draw (-1,.5) .. controls (-1.25,0.25) and (-1,0) .. (0,0);
\draw (-3,0) .. controls (-2, 0) and (-1.75,.25) .. (-2,.5);
\draw (-2,.5) .. controls (-2.25,.75) and (-3.75,.75) .. node {$<$} node [above = 1pt] {$t_y$} (-4,.5);
\draw (-4,.5) .. controls (-4.25,0.25) and (-4,0) .. (-3,0);
\draw (0,0) --node [below = 1pt] {$s_y$} node {$<$}  (-3,0);
\filldraw (0,0) node [below=1pt] {0} circle (2pt);
\filldraw (2,0) node [below=1pt] {1} circle (2pt);
\filldraw (-3,0) node[below=1pt] {$\infty$} circle (2pt);
\filldraw (4,0) node [below=1pt] {$x_0$} circle (2pt);
\filldraw (6,0) node [below=1pt] {$y_0$} circle (2pt);
\end{tikzpicture}
\end{center}

On the other hand, based at the tangential basepoint $(\overrightarrow{1 0}, \overrightarrow{\infty 0})$, $r_{(12)}$ would be comprised of the paths $s_x$ in the $x$- and $s_y$ in the $y$-coordinates respectively, since $(12)$ maps this basepoint to $( \overrightarrow{01}, \overrightarrow{0\infty})$, which is the exceptional case in which we replace the $x$-coordinate by $t_x^2 \circ s_x$ (a technical requirement in order to ensure that we end up with a group operation when we later compose paths in a specific way).

For a given $r_{\sigma}$ let $r_{\sigma}^{-1}$ denote the path which agrees with $r_{\sigma^{-1}}$ except that every $t_x$ and every $t_y$ appearing in $r_{\sigma^{-1}}$ is replaced by its complex conjugate path. This is exactly the path which runs in the opposite direction to $r_{\sigma}$, from suitable tangential basepoints so that the concatenation of the two paths is the trivial path. Notice that $r_{\sigma}^{-1}$ is not necessarily the same as any $r_{\tau}$ for any $\tau \in S_4.$ 

Now we can define a multiplication of paths of the form of $r_{\sigma}$ and $r_{\sigma}^{-1}$ where $\sigma \in S_4$, viewed as paths based at some fixed basepoint, in the obvious way. Once and for all fix the basepoint $(\overrightarrow{01}, \overrightarrow{01})$. Any of the $r_{\sigma}$ or $r_{\sigma}^{-1}$ which emanate from this basepoint terminate in some other basepoint of $\mathcal{B}$ and therefore it is possible to concatenate such paths. Let us denote the concatenation by $\odot$. 

Then $r_{\sigma} \odot r_{\sigma}^{-1}$ is the identity path, while $r_{\sigma} \odot r_{\tau}$ in general does not equal $r_{\tau \circ \sigma}$. For example, $r_{(12)} \neq r_{(12)}^{-1}$. In particular, 
$r_{(12)} \circ r_{(12)}$ is a loop about 1 in the $x$-component and a loop about $\infty$ in the $y$-component, both oriented positively, when based at $(\overrightarrow{01},\overrightarrow{01})$. 
By construction, the concatenation of paths does however respect the underlying $S_4$-action on $\mathcal{B},$ with $r_{\sigma}^{-1}$ mapping between the same basepoints as does $r_{\sigma^{-1}}$, and $r_{\tau\circ \sigma}$ running between the same basepoints as $r_{\sigma}\odot r_{\tau}$. 

Denote by $\mathcal{G}_{(01,01)}^{(4)}$ 
the group of paths which begin at $(\overrightarrow{01}, \overrightarrow{01})$, and which are generated as described above by the $\odot$ product of $r_{\sigma}$ and $r_{\sigma}^{-1}$ for $\sigma \in \{(12), (24), (23)\}$. We will show that the surjective homomorphism of $\mathcal{G}_{(01,01)}^{(4)}$ onto $S_4$ implicit in the construction factors through $PU(2,1; \mathbb{Z}[\rho]).$ 

In order to achieve this, we require a presentation for $PU(2, 1; \mathbb{Z}[\rho])$ given in \cite{FalbelParkerPicard}, which it is convenient to express in terms of explicit matrices which we derive from the Shiga generators of the monodromy representation listed above. 

First, as in \cite{FalbelParkerPicard} we introduce the involution
\[
R = \left[
\begin{array}{ccc}
0 & 0 & 1
\\
0 & -1 & 0
\\
1 & 0 & 0
\end{array}
\right]
\]
which will play an important role in what follows.

The following group elements will also be distinguished:
\begin{align*}
A_1: &= \left[
\begin{array}{ccc}
1 & 0 & 0
\\
0 & -\rho^2 & 0
\\
0 & 0 & 1
\end{array}
\right]
\qquad &
A_y: & = \left[
\begin{array}{ccc}
1 & 1 & \rho
\\
0 & -\rho^2 & \rho^2
\\
 0 & 0& 1
\end{array}
\right]
\\
A_0: &= \left[
\begin{array}{ccc}
1 & 0 & 0
\\
\rho & -\rho^2 & 0
\\
\rho & \rho & 1
\end{array}
\right]
\qquad &
A_{y,0}: & = \left[
\begin{array}{ccc}
0 & 0 & -\rho^2 - 1
\\
0 & 1 & 0
\\
-1-\rho^2 & 0 & 1-\rho^2
\end{array}
\right].
\end{align*}
These matrices were chosen because they satisfy $A_k^2 = g(\delta_k)$ for $k = 0,1,y$ and $\{y,0\}.$ Along with $R$, the $A_k$ generate $PU(2,1;\mathbb{Z}[\rho]).$ (In fact $R$, $A_0$ and $A_1$ suffice as generators.) 

The matrices $R_1 = A_1,$ $R_2 = R^{-1}A_yR$ and $R_3 = R^{-1}A_0R$ (where $R^{-1}=R$ and $[A_1,R]=I$) agree with generators for $PU(2,1;\mathbb{Z}[\rho])$ given in \cite{FalbelParkerPicard} (with the same $R_j$ notation), where the presentation
\begin{multline*}
PU(2, 1;\mathbb{Z}[\rho]) \simeq
<R_1,R_2,R_3\;|\;R_1^6 =R_2^6 = R_3^6= (R_3R_2R_1)^4 =I; 
\\
R_jR_{j+1}R_j = R_{j+1}R_jR_{j+1} \;{\text{ for }}\; j \in \{1,2,3\} \!\!\mod 3;  \\
R_1R_2R_3R_1 = R_3R_1R_2R_3>
\end{multline*}
is shown to hold.

\begin{lemma}
The mapping 
\label{l:1}
\[
\Upsilon : PU(2,1; \mathbb{Z}[\rho]) \rightarrow S_4
\]
defined by
\begin{align*}
R_1 & \mapsto s_1 & \leftrightarrow (12)
\\
R_2 & \mapsto s_2 & \leftrightarrow (24)
\\
R_3 & \mapsto s_2^{-1}s_3s_2 & \leftrightarrow (23)
\end{align*}
is a surjective homomorphism.
\end{lemma}
{\em Proof of Lemma:} One checks that the relations on the $R_j$ are compatible with those on the $s_k$. The surjectivity follows since the images $s_1,s_2$ and $s_2^{-1}s_3s_2$ form a generating set for $S_4.$
\hfill $\Box$

Because $R = [(R_1R_3R_3)^{-2}R_1R_2R_1^{-1}(R_1R_2R_3)^{-2}R_1R_2]^2 = (R_3R_1R_2)^2$ (cf. \cite{FalbelParkerPicard}), one checks that $\Upsilon: R \mapsto (12)(34)$.

Notice that in each of the generators of $\mathcal{G}_{(01,01)}^{(4)}$, in each occurrence, $t_x$ and $t_y$ are positively oriented. This is essentially why we have the
\begin{proposition}
\label{p:free}
$\mathcal{G}_{(01,01)}^{(4)}$ is a free group.
\end{proposition}
\begin{proof} Suppose instead that a relation $r_{(a_1b_1)} \odot \cdots \odot r_{(a_k b_k)} = 1$ (with $(a_j b_j) \in \{(12), (23), (24)\}$ for $j=1, \ldots, k$) exists among the generators of $\mathcal{G}_{(01,01)}^{(4)}$. Then by construction also
\[
(a_1b_1) \cdots (a_k b_k) = e
\]
in $S_4.$ Now we know that the relations among the transpositions $(12)$, $(23)$ and $(24)$ are generated by $(12)^2 = (23)^2 = (24)^2 = e$ and 
$((12)(23))^3 = ((12)(24))^3 = ((23)(24))^3 = e.$ One checks easily that none of the corresponding products in the $r_{(a_j b_j)}$ give the identity path in $\mathcal{G}_{(10,\infty0)}^{(4)}$. Indeed, each of $r_{(12)}$, $r_{(24)}$ and $r_{(23)}$ has infinite order; and of the other relevant paths to consider, 
the simplest is 
\[
(r_{(23)}\odot r_{(24)}) = (1, t_y^6)
\]
which is also of infinite order.
\end{proof}

\begin{corollary}
The mapping 
\[
\mathcal{T}: \mathcal{G}_{(01, 01)}^{(4)} \rightarrow PU(2,1;\mathbb{Z}[\rho])
\]
given by 
\begin{align*}
r_{(12)} &\mapsto R_1
\\
r_{(24)} & \mapsto R_2
\\
r_{(23)} & \mapsto R_3
\end{align*}
is a surjective homomorphism.
\end{corollary}
\begin{proof}
Since $\mathcal{G}_{(01,01)}^{(4)}$ is a free group the map induced by generators is a homomorphism; and it is surjective since $R_1, R_2$ and $R_3$ generate $PU(2, 1; \mathbb{Z}[\rho]).$ 
\end{proof}

Denote the kernel of the surjection of  $\mathcal{G}_{(01,01)}^{(4)}$ onto $S_4$ by $\Gamma_{(01,01)}^{(4)}.$ Notice that it consists of those elements of the fundamental group $\pi_1((\mathbb{P}^1\backslash\{0,1,\infty\})^2\backslash \Delta, (\overrightarrow{01}, \overrightarrow{01}))$ which lie in $\mathcal{G}_{(01,01)}^{(4)}.$

Recall that $\Gamma_1$ denotes the monodromy representation for $\mathcal{M}_{0,5}$, for which Shiga's matrix generators were recorded above. Then it again follows immediately that we have the 
\begin{corollary}
The mapping 
\[
\mathcal{U}: \Gamma_{(01,01)}^{(4)} \rightarrow \Gamma_1
\]
induced by $\mathcal{T}$ is a surjective homomorphism.
\end{corollary}

$\mathcal{U}$ can also be given very explicitly since one can check that (the isomorphic copy of) the monodromy representation $\Gamma_1$ which embeds into $PU(2,1;\mathbb{Z}[\rho])$ as given by the above presentation is generated by $R_1^2 = g(\delta_1), R_2^2 = R^{-1}g(\delta_y)R,R_3^2 =R^{-1}g(\delta_0)R$, $(R_1R_2)^3 = g(\delta_{y, \infty})$, $(R_1R_2R_3R_2^{-1})^2 = R^{-1}g(\delta_1)g(\delta_{y,0})R$
and $(R_2^2R_3R_2^{-1})^3 = R^{-1}g(\delta_0)g(\delta_{y,0})g(\delta_y)R.$

\subsection{Picard period polynomials} 
\label{s:2.2} In order to define Picard period polynomials by analogy with the usual (elliptic) period polynomials discussed in \S \ref{s:1.2}, we first assemble the relevant constructions and definitions.

Following \cite{Shiga}, define a Picard modular form (respectively a Picard meromorphic modular form) of weight $k$ relative to a subgroup $G$ of $PU(2,1;\mathbb{Z}[\rho])$ to be any holomorphic (respectively meromorphic) function $f(z_1,z_2)$ on the two dimensional complex ball $\mathbb{H}^2_{\mathbb{C}}$ which satisfies
\begin{equation}
\label{e:pmf}
f(g(z_1,z_2))
= (\det g)^{-k}
(g_{20}z_1+g_{21}z_2 +g_{22})^{3k}f(z_1,z_1)
\end{equation}
for any $g = (g_{ij})_{i,j = 0,1,2} \in G$. 
Here $g(z_1,z_2)$ is defined via the action of $g$ on $[\eta_0:\eta_1:\eta_2] \in \mathbb{P}^2(\mathbb{C})$ with the identifications $z_1 = \frac{\eta_0}{\eta_2}$ and $z_2 = \frac{\eta_1}{\eta_2}.$ As Shiga points out, one checks that 
\begin{equation}
\label{e:pj}
\frac{\partial(g(z_1,z_2))}{\partial(z_1,z_2)}
=
\frac{\det g}{(g_{20}z_1 + g_{21}z_2 + g_{22})^{3}} =:j_g(z_1,z_2)
.
\end{equation} 
In other words, (\ref{e:pmf}) is of the form 
\begin{equation}
\label{e:Jd}
f(z_1,z_2) = j_g^{k}(z_1,z_2) f(g(z_1,z_2))
\end{equation}
where $j_g$ denotes the Jacobian determinant given by (\ref{e:pj}). Note that the equations given here adhere to the conventions of \cite{FalbelParkerPicard} for the $PU(2,1;\mathbb{Z}[\rho])$ action on $\mathbb{H}^2_{\mathbb{C}}$ - i.e. we identify $PU(2,1;\mathbb{Z}[\rho])$ with $\Gamma$ in (\ref{e:FPPconv}) and label the coordinates on $\mathbb{H}_{\mathbb{C}}^2$ according to the ordering given in (\ref{e:ball}).

Next, we remark that the (left) action induced by the transpose of the inverse $^tg^{-1}$, multiplied by a cubed root of $\det g$, turns out to give a suitable action of $g= (g_{ij})_{i,j = 0,1,2} \in PU(2,1;\mathbb{Z}[\rho])$ on the polynomial ring $\mathbb{Z}[\rho][X_0,X_1,X_2]_{3n}$ of homogeneous polynomials in the variables $X_0, X_1$ and $X_2$ of degree $3n$ for some fixed $n\geq 1$ with coefficients in $\mathbb{Z}[\rho]$. Given a general element $A = (a_{ij})_{i,j  = 0,1,2}$ of $PU(2,1; \mathbb{Z}[\rho])$ the inverse is 
\[
A^{-1}
 = 
 \left[
 \begin{array}{ccc}
 \overline{a}_{22} & \overline{a}_{12} & \overline{a}_{02}
 \\
\overline{ a}_{21} & \overline{a}_{11} & \overline{a}_{01}
 \\
\overline{ a}_{20} & \overline{a}_{10} & \overline{a}_{00}
 \end{array}
 \right]
\]
-i.e. $A^{-1} = (b_{ij})_{i,j =0,1,2}$ where $b_{ij} = \overline{a}_{(2-j)(2-i)}$. Hence this action is given by:
\begin{equation}
\label{e:pa}
X_k \mapsto (\det g)^{1/3}\sum_{j=0}^2 \overline{g}_{(2-k)(2-j)}X_j
\end{equation}
for $k=0,1,2.$ Here $\det g$ is some sixth root of unity for any $g \in PU(2, 1;\mathbb{Z}[\rho])$ but notice that since we are working only with the action on homogeneous polynomials of degree divisible by 3, the choice of the cubed root is irrelevant. Also, since the sixth roots of unity lie in $\mathbb{Z}[\rho]$ we don't need to enlarge the ring of coefficients.

Suppose that $f(z_1,z_2)$ denotes a Picard modular form of weight $k$ relative to a subgroup $G$ of $PU(2,1;\mathbb{Z}[\rho]).$ By combining (\ref{e:pmf}), (\ref{e:pj}) and (\ref{e:pa}) one checks that 
\[
\underline{f}(z_1,z_2; X_0,X_1,X_2)dz_1\wedge dz_2:=
f(z_1,z_2)(z_1X_0+z_2X_1+X_2)^{3k-3}dz_1\wedge dz_2
\]
is invariant under the action of $G$. 

To obtain Picard modular forms for the full group $PU(2,1;\mathbb{Z}[\rho])$ and not merely subgroups thereof, we need to work a little harder. In \cite{RungePicardMod} Runge used a method he developed to determine rings of Siegel modular forms, to find a presentation of the ring of Picard modular forms for the full group $PU(2, 1; \mathbb{Z}[\rho])$, generalizing earlier work of Holzapfel \cite{Holzapfel} and Shiga \cite{Shiga}. In doing so, he introduced the notion of ``Picard type'': let $M$ denote an element of finite order in $Sp(3, \mathbb{Z})$ and let $\mathfrak{H}(M)$ be the connected complex submanifold of the Siegel upper half space $\mathfrak{H}_3$ consisting of those points which are fixed by the action of $M$. If $\Gamma_M$ denotes the centralizer of $M$ in $Sp(3, \mathbb{Z})$, then he shows that $\Gamma_M\backslash \mathfrak{H}(M)$ is the Picard moduli variety (in the sense that the coarse moduli space of Picard curves is dense therein - see Proposition 4 in \cite{RungePicardMod}), and $M$ is referred to as the {\em Picard type}. A holomorphic function $f$ on $\mathfrak{H}_3$ which for every
\[
\gamma = \left[
\begin{array}{cc}
A_{\gamma} & B_{\gamma}
\\
C_{\gamma} & D_{\gamma}
\end{array}
\right]
\]
in $\Gamma_{M}$ satisfies
\begin{equation}
\label{e:mPt}
f(\tau) = j_{\gamma}^k (\tau)\cdot f(\gamma(\tau))
\end{equation}
(where $j_{\gamma}(\tau)$ denotes the Jacobian determinant $C_{\gamma}\tau + D_{\gamma}$),  is called a Picard modular form of weight $k$ and Picard type $M$. (Here $\gamma$ has the usual action on $\tau \in \mathfrak{H}_3$ - namely $\gamma(\tau) = (A_{\gamma}\tau + B_{\gamma})\cdot (C_{\gamma} \tau+D_{\gamma})^{-1}$.)

Now it is known that the period matrix 
\begin{equation}
\label{e:pm}
\Omega(z_1, z_2)
=
\left[ 
\begin{array}{ccc}
\frac{2\rho^2 z_1 + z_2^2}{1-\rho} & \rho^2 z_2 & \frac{\rho^2 z_1 - \rho z_2^2}{\rho - 1}
\\
\rho^2 z_2 & -\rho^2 & z_2
\\
\frac{\rho^2 z_1 - \rho z_2^2}{\rho-1} & z_2 & \frac{2z_1+z_2^2}{(1-\rho)\rho}
\end{array}
\right]
\end{equation}
gives a mapping, say $\Psi$, from $\mathbb{H}_{\mathbb{C}}^2$ with the action of $PU(2, 1; \mathbb{Z}[\rho])$ into $\mathfrak{H}_3$ with the action of $Sp(3, \mathbb{Z})$.  In \cite{Holzapfel}, Holzapfel gives an explicit correspondence between $PU(2, 1; \mathbb{Z}[\rho])$ and a certain subgroup of $Sp(3, \mathbb{Z})$. This correspondence is defined in such a way that $\Psi$ carries the $PU(2,1; \mathbb{Z}[\rho])$ action to the $Sp(3, \mathbb{Z})$ action. (See Lemma 2.27 in \cite{Holzapfel}). Under these correspondences it then follows that the defining equations of Picard modular forms given by (\ref{e:Jd}) and (\ref{e:mPt}) coincide. 

 In Runge's formalism, when the Picard type is 
\[
M = \left[
\begin{array}{rrrrrr}
0 & 0& 0& -1 & 0 & 0
\\
0 & 0 & 0 & 0 & -1 & 0 
\\
0 & 0 & -1 & 0 & 0 & 1
\\
1 & 0 & 0 & -1 & 0 & 0 
\\
0 & 1 & 0 & 0 & -1 & 0 
\\
0 & 0 & -1 & 0 & 0 & 0 
\end{array}
\right]
\]
it turns out that $\Gamma_M \simeq PU(2,1; \mathbb{Z}[\rho])$. Runge found the ring of Picard modular forms of type $M$ is a polynomial ring in certain homogeneous polynomials in the theta constants
\begin{equation}
\label{e:tc}
f_{\mathbf{k}}(\Omega):= \theta
 \left[
 \begin{array}{c}
\mathbf{k}
\\
\mathbf{0}
\end{array}
\right]  (0, 2\Omega) := \sum_{\mathbf{n} \in \mathbb{Z}^3}
\exp( 2 \pi i\;  {^t(\mathbf{n}+\mathbf{k})}\Omega (\mathbf{n}+\mathbf{k})) 
\end{equation}
where $\Omega \in \mathfrak{H}_3$ and $\mathbf{k}$ is one of 
\[
\mathbf{1} := \left[ \begin{array}{c}
                    1\\0\\0
                    \end{array}\right], \;\;
\mathbf{2} := \left[ \begin{array}{c}
                    0\\1\\0
                    \end{array}\right], \;\;
{\text{ or }}\; \;
\mathbf{3} := \left[\begin{array}{c}
                    1\\1\\0
                    \end{array}\right].
                    \] 
                    See Appendix \ref{a:Runge} for two of the  explicit generating polynomials.  
  
Thanks to the period mapping of (\ref{e:pm}), these functions may be regarded as functions of the ball $\mathbb{H}_{\mathbb{C}}^2$. 
Then if $f(z_1,z_2)$ is a Picard modular form of type $M$ for the entire group $PU(2, 1;\mathbb{Z}[\rho])$, as above
\[
\underline{f}(z_1,z_2; X_0,X_1,X_2)dz_1\wedge dz_2:=
f(z_1,z_2)(z_1X_0+z_2X_1+X_2)^{3k-3}dz_1\wedge dz_2
\]
will be invariant under the action of $PU(2,1;\mathbb{Z}[\rho])$.

Integrating this form over a certain domain $D$ in $\mathbb{H}_{\mathbb{C}}^2$ will give the Picard period polynomials, so we turn next to a description of $D$. 


To begin, we motivate the claim made above that $R$ is an analogue of the matrix $S$ of $PSL(2, \mathbb{Z}).$ As in \cite{FalbelParkerPicard} let 
\[
P = \left[
\begin{array}{ccc}
1 & 1& \rho
\\
0 & \rho & -\rho
\\
0 & 0 & 1
\end{array}
\right]
\]
which is also an element of $PU(2,1; \mathbb{Z}[\rho])$. The subgroup generated by $P$ and $R$ is a non-faithful representation of $PSL(2, \mathbb{Z})$ in $PU(2,1)$, from which $PU(2,1;\mathbb{Z}[\rho])$ may be recovered by adjoining an elliptic element of order 6, for example $R_1$. (See Proposition 5.10 in \cite{FalbelParkerPicard}.) The involution $R$ is the image of the involution $S$ of $PSL(2, \mathbb{Z})$ in this representation. Also following \cite{FalbelParkerPicard}, denote the point $[1:0:0]$ at infinity for $\mathbb{H}_{\mathbb{C}}^2$ by $q_{\infty}$ and define the isometric sphere $\mathcal{S}_R$ of $R$ as follows: 
\[
\mathcal{S}_R:=\{z \in \mathbb{H}_{\mathbb{C}}^2 \;:\;
|<q_{\infty},z>| = |<q_{\infty},Rz>|
\}.
\] 
The action of $R$ on $\mathcal{S}_R$ is most easily understood via geographical coordinates, which are the natural coordinates to emphasize properties of $\mathcal{S}_R$ (and other similarly defined isometric spheres) shown in work of Mostow and Goldman to be satisfied by so-called bisectors, of which these are examples. See \cite{FalbelParkerPicard} for more details. A straightforward calculation shows that on $\mathcal{S}_R$ we can take $z_1 = -e^{i\theta}$ where $\theta \in [-\pi/2, \pi/2].$ It's then convenient to set
\[
z_2 = re^{i\alpha + i\theta/2}.
\]
With this choice, the points of $\mathcal{S}_R$ are those for which $r \in [-\sqrt{2 \cos \theta}, \sqrt{2 \cos \theta}]$ and $\alpha \in [-\pi/2, \pi/2)$.

Now a trivial calculation shows that $R$ maps $\mathcal{S}_R$ to itself by sending the point with geographical coordinates $(r, \theta, \alpha)$ to the point $(r, -\theta, \alpha).$ The points of $\mathcal{S}_R$ for which $\theta = 0$ are therefore fixed by the $R$-action. Denote this set of points by $i_3,$ since it is the analogue of the fixed point $i \in \mathfrak{H}$ of $S$. 

The so-called horospherical coordinates on $\mathbb{H}_{\mathbb{C}}^2$ will also be useful: again as in  \cite{FalbelParkerPicard} let $\mathcal{N}$ denote the Heisenberg group $\mathbb{C}\times \mathbb{R}$ with group law given by
\[
(a_1,t_1)\cdot (a_2,t_2) = (a_1+a_2,\;t_1+t_2 - \Im (a_1\overline{a}_2)).
\]
Then $\mathcal{N}\times \mathbb{R}^{+}$ parameterizes $\mathbb{H}_{\mathbb{C}}^2$ as follows:
\[
(a,t,u)\mapsto \left[\frac{-|a|^2 - u+it}{2}\;:\; a\;:\; 1\right]
\]
In these coordinates, one sees that
\[
\mathcal{S}_R = \{(a,t,u)\in \mathcal{N}\times \mathbb{R}^+\;:\; ||a|^2+u+it| = 2\}
\]
and then it's easy to show that the $R$-action on $\mathbb{H}_{\mathbb{C}}^2$ is a reflection across this 
isometric sphere.

This is a precise analogue of the action of $S$ on $PSL(2, \mathbb{Z})$ with respect to the geodesic arc $|z|=1$ in the upper half space $\mathfrak{H}$: indeed, define the Hermitian form $<w,z> = \overline{w}z$ for $w,z$ in $\mathfrak{H}$, consider 1 as a cusp of $\mathfrak{H}$ and then set
\[
\mathcal{U}_S:= \{z \in \mathfrak{H}\;:\; <1,z> = <1,Sz>\}.
\]
It's clear that $\mathcal{U}_S$ is the upper half of the circle $|z| = 1$ and that $S$, which acts on $\mathfrak{H}$ by $z \mapsto -\frac{1}{z}$ is a reflection across this arc. 

Not only is $i \in \mathcal{U}_S$ fixed by $S$ in the same way that the points of $i_3 \in \mathcal{S}_R$ are fixed by $R$; but the $S$-action interchanges the center 0 of $\mathcal{U}_S$ and $i\infty$, while the $R$-action switches the center $(0,0,0)$ (in horospherical coordinates) of the isometric sphere $\mathcal{S}_R$ and $q_\infty.$ 

Next we consider fixed points of the generators in  the generating set $\{R_1,R_2,R_3\}$. In the $PSL(2, \mathbb{Z})$ case, $T$ fixes $i\infty$ and the path of integration for the period polynomials is comprised of the path from $i\infty$ (the fixed point of $T$) to $i$ (the fixed point of $S$ on $\mathcal{U}_S$), together with the image thereof under $S$ itself. In the $PU(2,1;\mathbb{Z}[\rho])$ setting, one checks that $R_1$ fixes both $(0,0,0) = [0:0:1]$ and $q_\infty$, and shares with $R$ the fixed point $[-1:0:1] \in i_3.$ $R_2$ fixes $[0:0:1]$ and $[-1:-\rho:1] \in i_3$ while $R_3$ fixes $q_{\infty}$ and $[-1:1:1]\in i_3.$ Within $i_3$ consider the path $j_3$ comprising the geodesic of $\mathbb{H}_{\mathbb{C}}^2$ from the fixed point $[-1: -\rho: 1]$ of $R_2$ to the fixed point $[-1:0:1]$ of $R_1$, followed by the geodesic from $[-1:0:1]$ to the fixed point $[-1:1:1]$ of $R_2$. In summary, $j_3$ comprises the geodesics running from $[-1:-\rho:1]$ to $[-1: 0 :1]$ to $[-1:1:1]$ oriented in the order given. 

Now define the domain $D$ in $\mathbb{H}_{\mathbb{C}}^2$, which is to serve as the analogue of the line from $i\infty$ through $i$ to $0$ in $\mathbb{H},$ to consist of the union of all geodesics from $q_{\infty}$ to the points of $j_3$, together with the union of the images of these lines under the $R$-action, each oriented from a point of $j_3$ to $[0:0:1]$. To describe this hyperplane more precisely, note that in horospherical coordinates, the geodesic from a point $(a,t,u)\in \mathcal{N}\times \mathbb{R}^+$ to $q_\infty$ is 
\[
l_{(a,t,u)} = \{(a,t,w) \in \mathcal{N}\times \mathbb{R}^+\;:\; w \geq u\}
\]
(cf. \cite{FalbelParkerPicard}). In these coordinates we have 
\begin{align*}
 j_3 &=  \{(e^{-i\pi/3}(1-b) ,0,2-(1-b)^2)\in \mathcal{N}\times \mathbb{R}^+\;:\;
b \in [0,1]\} 
\\
 &\cup \{(c ,0,2-c^2)\in \mathcal{N}\times \mathbb{R}^+
 \;:\;
c \in [0,1]\}
,
\end{align*}
so for any point $(a,0, w)$ on $j_3,$ one computes in projective coordinates
\[
l_{(a,0,w)} = \{[-1-\frac{u}{2}: a:1]\;:\; u \in [0,\infty)\},
\]
the image of which under $R$ is
\[
R* l_{(a,0,w)} = \left\{\left[-\frac{2}{2+u}:\frac{2 a}{2+u}:1\right]\;:\; u \in [0,\infty)\right\}.
\]
 Therefore, if $j_3$ is oriented as above,
\[
D = \bigcup_{\substack{ (a,0,w) \in j_3 \\
q \in [0, \infty)}} 
l_{(a,0,w+q)}
\;\;\cup 
\bigcup_{
\substack{ (a,0,w) \in j_3 \\
q \in [0, \infty)}} 
R*l_{(a,0,w+q)}
\]            
may be visualized (in affine coordinates) as two sides of the surface of a triangular prism in $\mathbb{R} \times \mathbb{C}$ which extends from infinity along the negative real axis towards the omitted vertex $(0,0)$.

Finally we define the Picard period polynomial associated to the Picard modular form $f$ of type $M$ as
\[
P_f(X_0,X_1,X_2):=
{\int}_{{\!\!D}} \;\;f(z_1,z_2)(z_1X_0+z_2X_1+X_2)^{3k-3}\;dz_1\wedge dz_2.
\]

First we must establish the
\\
\begin{proposition} 
\label{p:conv}
{The integral $P_f(X_0,X_1,X_2)$ converges.}
\end{proposition}

\begin{proof} Recall that the domain $D$ of integration is the union of the set (say $L$) of geodesics from the point at infinity $q_{\infty}$ to the points of $j_3$, together with the (suitably re-oriented) image thereof under the action of $R$ (for which we write $R*L$). 

Then, should $P_f(X_0,X_1,X_2)$ converge, it may be written as a sum of integrals over the two respective regions $L$ and $R*L$. Consider the integral over $R*L$:
\begin{align*}
&{\int}_{{\!\!R*L}} \;\;f(z_1,z_2)(z_1X_0+z_2X_1+X_2)^{3k-3}\;dz_1\wedge dz_2.
\\
&=
-\int_{{\!\! L}}\;\; f(R^{-1}(z_1,z_2))
(R^{-1}(z_1)X_0+R^{-1}(z_2)X_1 + X_2)^{3k-3}
\frac{\partial R^{-1}(z_1,z_2)}{\partial (z_1,z_2)}
\;dz_1\wedge dz_2
\end{align*}
where $R^{-1}(z_1)$ and $R^{-1}(z_2)$ is shorthand for the action of $R^{-1}$ on each of $z_1$ and $z_2$ coming from the action on $(z_1,z_2)$ as before, and the sign change comes from the reorientation.  
But $ f(z_1,z_2)(z_1X_0+z_2X_1+X_2)^{3k-3}\;dz_1\wedge dz_2$ is invariant under the action of $R = R^{-1}$. Therefore  
\begin{align*}
& f(R^{-1}(z_1,z_2))
(R^{-1}(z_1)X_0+R^{-1}(z_2)X_1 + X_2)^{3k-3}
\frac{\partial R^{-1}(z_1,z_2)}{\partial (z_1,z_2)}
\;dz_1\wedge dz_2 \\
&
= f(z_1,z_2)(z_1X_2-z_2X_1+X_0)^{3k-3}\;dz_1\wedge dz_2
\end{align*}
since the action of $R$ on $(X_0,X_1,X_2)$ gives $(X_2, -X_1, X_0)$. By the above expression for the integral over $R*L$ in terms of an integral over $L$, we can conclude that the integral over $R*L$ converges if and only if the integral over $L$ converges. 

Now $R*q_\infty = [0 : 0 : 1]$ and one checks that at this point,
\[
\theta
 \left[
 \begin{array}{c}
\mathbf{k}
\\
\mathbf{0}
\end{array}
\right]  (0, 2\Omega(0,0)) = \sum_{(n_0,n_1,n_2) \in \mathbb{Z}^3}
\exp( 2 \pi i\;  (-1)(n_1+\delta_{k2}1/2)^2\rho^2)) 
\]
where $\delta_{k2}$ equals 1 when $\mathbf{k} = \mathbf{2}$ and is 0 otherwise. But this sum is finite since $\Im (\rho^2)<0.$ Because $R*L \cup R*q_{\infty}$ is compact one sees that the integral over $R*L$ converges, and thus the integral over the entire domain $D$ does too. 
\end{proof}

In order to obtain the exact analogue of Theorem \ref{t:1} in this context, we would require information pertaining to the image of $D$ under the uniformization map $\Psi : \mathbb{H}_{\mathbb{C}}^2\rightarrow \mathcal{M}_{0,5}$. Shiga found this map explicitly in terms of theta functions (cf. Proposition I-3 of \cite{Shiga}), but it is very difficult to compute even in the simplest cases. Instead of taking this approach, we will prove directly (working in $\mathfrak{H}_{\mathbb{C}}^2$) that each relation on $PU(2,1;\mathbb{Z}[\rho])$ gives rise to a relation on the period polynomials, from which it follows immediately that any contractible path in $\mathcal{G}_{(01,01)}^{(4)}$ also does.







The presentation of $PU(2,1;\mathbb{Z}[\rho])$ which is most convenient to use here comes from the representation of $PSL(2, \mathbb{Z})$ given by $S \mapsto R$ and $T \mapsto P$; to which $R_1$ is adjoined (again see Proposition 5.10 in \cite{FalbelParkerPicard}):
\[
PU(2,1;\mathbb{Z}[\rho])
=
<R,P,R_1| R^2 = (RP)^6 = R_1^6 = [R_1,R] = PR_1^{-1}P^{-1}R_1^{-1}P = I>
\]




\begin{theorem}
\label{t:2}
The period polynomial $P_f(X_0,X_1,X_2)$ for the Picard modular form $f$ of type $M$ and weight $k$ satisfies the following relations corresponding to defining relations on $PU(2,1;\mathbb{Z}[\rho]):$
\begin{equation}
\label{e:R^2}
P_f(X_0,X_1,X_2) = -P_f(X_2,-X_1,X_0)
\end{equation}
corresponding to $R^2 = I;$

\begin{multline}
P_f(X_0,X_1,X_2)+(-\rho^2)^{k-1} P_f(X_0,-\rho X_1,X_2)+\rho^{k-1} P_f(X_0,\rho^2 X_1,X_2)
\\
\label{e:r1^6}
 +(-1)^{k-1}P_f(X_0,-X_1,X_2)+\rho^{2k-2} P_f(X_0,\rho X_1,X_2)+(-\rho)^{k-1} P_f(X_0,-\rho^2 X_1,X_2) =0
\end{multline}
corresponding to $R_1^6 = I$; 

\begin{multline}
\label{e:PM}
P_f(X_0,X_1,X_2)+\rho^{k-1} P_f(\rho^2 X_0+X_1+X_2, \rho^2 X_0-\rho^2 X_1, X_0) 
\\
 + \rho^{2k-2} P_f(\rho^2 X_2,\rho^2 X_2 - X_1, \rho^2 X_0+X_1+X_2) =0
\end{multline}
corresponding to $(RP)^3 = I;$

\begin{equation}
\label{e:[RR1]}
P_f(X_2, \rho X_1, X_0)+P_f(X_0,-\rho X_1,X_2) = 0
\end{equation}
coming from $[R,R_1] = I$; and the pair of relations
\begin{multline}
\label{e:pent1}
P_f(X_0,X_1,X_2)+\rho^{k-1} P_f(X_0,-\rho^2 X_0+\rho^2 X_1, \rho^2 X_0+X_1+X_2)
\\
=
(-\rho^2)^{k-1}P_f(X_0,\rho X_0 - \rho X_1, \rho^2 X_0+X_1+X_2)
\end{multline}
and
\begin{multline}
\label{e:pent2}
P_f(X_0,X_1,X_2)+\rho^{2k-2} P_f(X_0, X_0+\rho X_1, \rho X_0-\rho X_1+X_2)\\
=
(-\rho)^{k-1}P_f(X_0,-\rho X_0 - \rho^2 X_1, \rho X_0-\rho X_1+X_2)
\end{multline}
coming from $PR_1^{-1}P^{-1}R_1^{-1}P = I.$

\end{theorem}

\begin{proof} 
Since $R$ maps $D$ to itself but reverses each of the paths in the variable $z_1$ from $q_{\infty}$ to a point of $j_3$ to $[0:0:1]$, 
\begin{align*}
&P_f(X_0,X_1,X_2) 
\\
&= 
-\int_{\!\!R(D)} \;\;f(z_1,z_2)(z_1X_0+z_2X_1+X_2)^{3k-3}\;dz_1\wedge dz_2
 \\
 &= 
 -\int_{\!\! D} \;\;f(R^{-1}(z_1,z_2))((R^{-1}z_1)X_0+(R^{-1}z_2)X_1+X_2)^{3k-3}\;dR^{-1}z_1\wedge dR^{-1}z_2
 \\
 &=
 - \int_{\!\!D} \;\;f(z_1,z_2)(z_1X_0+z_2(-X_1)+X_2)^{3k-3}\;dz_1\wedge dz_2
 \\
 &= -P_f(X_0,-X_1,X_2),
 \end{align*}
using the invariance of the differential form under the action of $R^{-1} = R$, along with the $R$-action on $(X_0,X_1,X_2)$ from (\ref{e:pa}).

Next, recall that $R_1$ has order 6. From the definition it's clear that $R_1$ acts on $D$ by rotation of the coordinate $z_2$ of a point $[z_1:z_2:1]\in D$ through $\frac{2\pi i \cdot 5}{6} = -\rho$ - i.e. a rotation of the triangular region of which $D$ is the boundary by $\pi/3$ in the negative direction. It follows that the union of the images $R_1(D) \cup R_1^2(D) \cup R_1^3(D) \cup R_1^4(D) \cup R_1^5(D)$ taken with the reverse orientation of the variable $z_2$, is homotopic to $D$. Consequently
\begin{multline*}
0 = \int_D + \int_{R_1(D)} + \int_{R_1^2(D)}
+ \int_{R_1^3(D)} + \int_{R_1^4(D)} 
\\
+ \int_{R_1^5(D)}  \;\;f(z_1,z_2)(z_1X_0+z_2X_1+X_2)^{3k-3}\;dz_1\wedge dz_2
\end{multline*}
and the second equation in the theorem follows as before. 

To prove (\ref{e:PM}) we will work with the path $\gamma_\beta$ in $D$ consisting of the geodesic from $q_{\infty}$ to some fixed point $\beta$ of $j_3$ together with the path from $\beta$ to $[0:0:1]$ which is the image thereof under $R$ (with reversed orientation) and show that $\gamma_{\beta} \cup (RP)\gamma_{\beta} \cup (RP)^2 \gamma_{\beta}$ is a (contractible) triangle in $\mathbb{H}_{\mathbb{C}}^2$. The claim follows then from the fact that
\[
D = \bigcup_{\beta \in j_3} \gamma_{\beta}.
\]
First observe that $RP$ acts as follows on the cusps of $D:$
\[
q_\infty \;\stackrel{RP}{\mapsto} \;[0:0:1] \;\stackrel{RP}{\mapsto} \;[\rho^2:1:1] \;\stackrel{RP}{\mapsto}\;q_{\infty}.
\]
But then any $\gamma_{\beta}$ is mapped by $RP$ to a path from $[0:0:1]$ to $[\rho^2:1:1]$, which in turn maps to a path from $[\rho^2:1:1]$ to $q_{\infty}.$ Since each of these is comprised of geodesics because the elements of $PU(2,1;\mathbb{Z}[\rho])$ preserve the topology on $\mathbb{H}_{\mathbb{C}}^2$, these paths are mutually non-intersecting, and therefore 
they form the boundary of a triangle. Because the integral giving the Picard period polynomial is defined at the cusps, the integral over this triangle is zero. Repeating this for each $\beta$ in $j_3$, it follows that 
\begin{align*}
0 &=
\int_D + \int_{RP(D)} + \int_{(RP)^2(D)} 
\;\;f(z_1,z_2)(z_1X_0+z_2X_1+X_2)^{3k-3}\;dz_1\wedge dz_2
\end{align*}
from which a similar calculation as above yields (\ref{e:PM}).

Next consider the action of $R$, $R_1 \circ R$,  $R^{-1}\circ R_1\circ R$ and $R_1^{-1}\circ R^{-1} \circ R_1 \circ R$ successively on $D$. As discussed above the involution $R$ reflects $D$ about $j_3$. Then $R_1$ rotates that image by $-\frac{\pi}{3}$; and this 2-simplex is in turn reflected about the image therein of $j_3$. $R_1^{-1}$ finally acts by a rotation about $z_2 = 0$ of $\frac{\pi}{3}$ so that the final image is $D$ itself. Above we established the equation (\ref{e:R^2}) and here the reflection produced by $R$ on $(R_1\circ R)(D)$ similarly gives
\begin{multline}
0=\int_{(R_1 R)(D)} \;\;f(z_1,z_2)(z_1X_0+z_2X_1+X_2)^{3k-3}\;dz_1\wedge dz_2
\\+
\int_{(RR_1 R)(D)} \;\;f(z_1,z_2)(z_1X_0+z_2X_1+X_2)^{3k-3}\;dz_1\wedge dz_2
\end{multline}
from which (\ref{e:[RR1]}) follows. 

Finally, in considering the palindromic relation \begin{equation}
\label{e:pent}
PR_1^{-1}P^{-1}R_1^{-1}R=I,
\end{equation}
 recall the idea used in the proof of Proposition \ref{p:conv}: as before if $L$ denotes the points of $D$ lying on geodesics from $q_{\infty}$ to points of $j_3,$ 
then as explained in that proof, we can write
\begin{multline}
P_f(X_0,X_1,X_2) = \int_L \;\;f(z_1,z_2)(z_1X_0+z_2X_1+X_2)^{3k-3}\;dz_1\wedge dz_2
\\
-
\int_L  \;\;f(z_1,z_2)(z_1X_2-z_2X_1+X_0)^{3k-3}\;dz_1\wedge dz_2
\end{multline}
Notice that here, for any $A \in PU(2, 1; \mathbb{Z}[\rho]),$ if we compute
\[
\int_{A(D)}\;\;f(z_1,z_2)(z_1X_2-z_2X_1+X_0)^{3k-3}\;dz_1\wedge dz_2
\]
by viewing it as 
\[
\int_{A(L)} + \int_{A(RL)}\;\;f(z_1,z_2)(z_1X_0+z_2X_1+X_2)^{3k-3}\;dz_1\wedge dz_2,
\]
we still obtain 
\begin{multline}
\int_{L}\;\;f(z_1,z_2)(z_1A(X_0)+z_2A(X_1)+A(X_2))^{3k-3}\;dz_1\wedge dz_2
\\ -
\int_L \;\;f(z_1,z_2)(z_1A(X_2)-z_2A(X_1)+A(X_0))^{3k-3}\;dz_1\wedge dz_2
\end{multline}
using the invariance of $f(z_1,z_2)(z_1X_0+z_2X_1+X_2)^{3k-3}\;dz_1\wedge dz_2$ under $PU(2, 1;\mathbb{Z}[\rho])$ because the action on $(X_0,X_1,X_2)$ used in the second integral is that of $(R^{-1}A^{-1})^{-1} = AR.$ (Indeed, set $(RX_0,RX_1,RX_2) = (A^{-1}Y_0,A^{-1}Y_1, A^{-1}Y_2)$, use the $A$-invariance of the 2-form to get the expression  $f(z_1,z_2)(z_1Y_0+z_2Y_1+Y_2)^{3k-3}\;dz_1\wedge dz_2$ and finally replace $Y_j$ by $(AR)X_j$ for $j=0,1,2.$)

Hence we need only consider the images of $L$ under 
 the successive factors of the relation (\ref{e:pent}), as opposed to considering the entire $D$. 

Denote by $m_{\alpha}$ the geodesic from $q_{\infty}$ to a point $\alpha$ of $j_3.$ It turns out that $P$, $R_1^{-1}$ and $P^{-1}$ act (successively) on certain  specific geodesics $m_{\alpha}$ as follows:
\[
m_{[-1:-\rho:1]} \;\stackrel{P}{\mapsto} \;
m_{[-1:1:1]} \; \stackrel{R_1^{-1}}{\mapsto} \;
m_{[-1:\rho:1]}\;\stackrel{P^{-1}}{\mapsto}
\;
m_{[\rho^2:0:1]} \; \stackrel{R_1^{-1}}{\mapsto} \;
m_{[\rho^2:0:1]} \;\stackrel{P}{\mapsto}\;
m_{[-1:-\rho:1]}
\]
and 
\[
m_{[-1:1:1]} \;\stackrel{P}{\mapsto} \;
m_{[\rho:0:1]} \; \stackrel{R_1^{-1}}{\mapsto} \;
m_{[\rho:0:1]}\;\stackrel{P^{-1}}{\mapsto}
\;
m_{[-1:1:1]} \; \stackrel{R_1^{-1}}{\mapsto} \;
m_{[-1:-\rho:1]} \;\stackrel{P}{\mapsto}\;
m_{[-1:1:1]}
\]
To ease the discussion, let us refer to the geodesic $m_{[-1:1:1]}$ and its image in each successive image of $L$ as the {\em leading edge} of the 2-simplex, and similarly refer to $m_{[-1:-\rho:1]}$ and its images as the {\em trailing edges}. 
 
With further elementary calculations one sees that $P$ acts by translating $L$ so that the geodesic $m_{[-1:-\rho:1]}$ lands on $m_{[-1:1:1]}$ and $m_{[-1:1:1]}$ is moved to $m_{[\rho:0:1]}$. Next  $R_1^{-1}$ acts as a rotation about the geodesic $m_{[\rho:0:1]}$; and $P^{-1}$ translates this rotated image of $L$ so that the leading edge lands on $m_{[-1:1:1]}$ and the trailing edge becomes $m_{[\rho^2:0:1]}$. Another rotation under the action of $R_1^{-1}$ follows - this time about $m_{[\rho^2:0:1]}$. This maps the leading edge to $m_{[-1:-\rho:1]}$ and finally $P$ gives a translation onto the original $L$.

The following diagram is a schematic representing the image of $j_3$ under this succession of transformations, in which the edges marked $H$ are the images of the path from $[-1:-\rho:1]$ to $[-1:0:1]$:
 
 
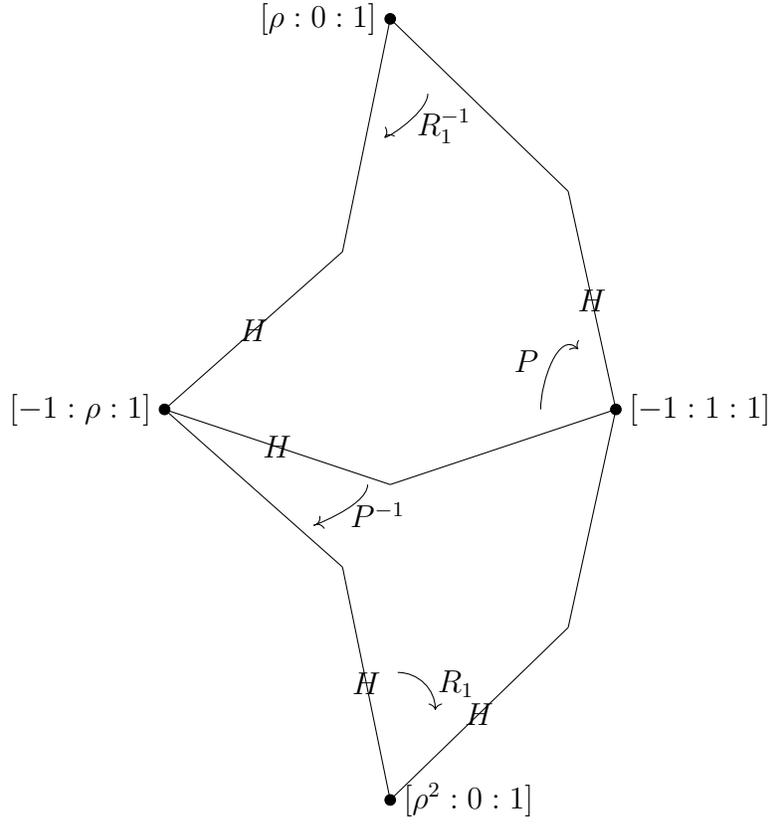
\begin{center}
 \begin{tikzpicture}
 \centering
 \filldraw (0,0) node [left=1pt] {$[-1:\rho:1]$} circle (2pt);
\filldraw (6,0) node [right=1pt] {$[-1:1:1]$} circle (2pt);
\draw (0,0) edge node  {$H$} (3,-1);
\draw (3,-1) edge (6,0);
\draw (6,0) edge node {$H$} (5.366,2.9019);
\draw (5.366,2.9019) edge (3, 5.196);
\draw (6,0) edge  (5.366,-2.9019);
\draw (5.366,-2.9019) edge node {$H$} (3, -5.196);
\filldraw (3,5.196) node [left=1pt] {$[\rho:0:1]$} circle (2pt);
\filldraw (3,-5.196) node [right=1pt] {$[\rho^2:0:1]$} circle (2pt);
\draw (3, 5.196) edge (2.366025, 2.09807);
\draw (2.366025, 2.09807) edge node {$H$} (0,0);
\draw (3, -5.196) edge node {$H$} (2.366025, -2.09807);
\draw (2.366025, -2.09807) edge (0,0);
\draw[-<] (2.7,-1) to [out = 270, in = 200] node [right=1pt] {$P^{-1}$} (2.1, -1.5);
\draw[<-] (3.6, -4) to [out = 90, in = 0] node [right=1pt] {$R_1$} (3.1,-3.5);
\draw[->] (5,0) to [out = 90, in = 135] node [left=1pt] {$P$} (5.5,0.8);
\draw[-<] (3.5, 4.2) to [out = 270, in = 225] node [right=1pt] {$R_1^{-1}$} (3,3.7);\end{tikzpicture}
 \end{center}

As is clear from the diagram above, integrating along the 2-simplex $L$ followed by $P(L)$ is the same as integrating along $(R_1^{-1}P)(L)$. This fact gives (\ref{e:pent1}). Similarly, integrating along $L$ reversing the orientation in the variable $z_2$ (i.e. from $m_{[-1:1:1]}$ to $m_{[-1:-\rho:1]}$) then along $P^{-1}(L) = (R_1^{-1}P^{-1}R_1^{-1}P)(L)$ (with the same reversal of the $z_2$-orientation) is the same as integrating along $(P^{-1}R_1^{-1}P)(L)$ again with similar reversal of orientation. From this fact the second equation (\ref{e:pent2}) coming from this relation can be shown to hold as before.
 \end{proof}

{\bf Remarks.} The equation (\ref{e:[RR1]}) is clearly equivalent to (\ref{e:R^2}).

Equation (\ref{e:PM}) is the Picard period polynomial analogue of the Manin equation (\ref{e:DA2}). Notice that our proof follows the classical approach given for example in \cite{LangIntroModForms} of integrating over triangles in the universal covering space rather than working in the moduli space directly as we were able to do in the proof of Theorem \ref{t:1}.

Equations (\ref{e:pent1}) and (\ref{e:pent2}) are hexagonal-type equations similar to the Manin relation (\ref{e:PM}) but which come out of the pentagonal equation (\ref{e:pent}) on $PU(2,1; \mathbb{Z}[\rho]).$ A geometric reason that we would not expect a pentagonal equation to be faithfully represented on the symmetry relations satisfied by the $P_f(X_0,X_1,X_2)$ is that the orbit of any tangential basepoint for a pentagonal path $\gamma_5$ in $\mathcal{M}_{0,5}$ under the induced $S_4$ action on tangential basepoints does not include all of the other tangential basepoints which occur in $\gamma_5:$ indeed, as is explained in \cite{Ihara:braids}, an element of $S_5$ of order 5 permutes these basepoints.

The computations needed for the proof show that the 2-simplex $P(D)$ and the 2-simplex $(P^{-1}R_1^{-1}P)(D)$ are complex conjugates of one another. This gives the following further (non-trivial) relation on period polynomials for Picard modular forms:
\[
P_f(X_0,-\rho^2X_0+\rho^2 X_1 , \rho^2X_0 +X_1+X_2)
=\overline{P_f}(X_0, -\rho X_0 -\rho^2 X_1, \rho X_0 - \rho X_1 + X_2)
\]
where $\overline{P_f}$ denotes complex conjugation.

Finally we remark that since we know that every relation on $PU(2,1;\mathbb{Z}[\rho])$ yields a relation on the period polynomials for Picard modular forms, the relations coming from the presentation of $PU(2,1;\mathbb{Z}[\rho])$ using the generators $R_1,R_2$ and $R_3$ do so as well. For example, since $R_3 = PR_1^{-1}$, one can deduce $R_3^6 = I$  from $R_1^6=I$ along with repeated application of $PR_1^{-1}P^{-1}R_1^{-1}P = I$ - indeed, $P^{-1}R_1=PR_1^{-1}P^{-1}$ so 
\begin{align*}
I &= P^{-1}R_1^6 P = (P^{-1}R_1)R_1^5 P
\\
&= (PR_1^{-1}P^{-1})R_1^5 P = (PR_1^{-1})(PR_1^{-1}P^{-1})R_1^4 P
\\
&= \cdots = (PR_1^{-1})^6 = R_3^6.
\end{align*}
Consequently, through repeated application of \ref{e:pent1} and \ref{e:pent2} to \ref{e:r1^6}, the equation corresponding to $R_3^6=I$ would result, namely 
\begin{align*}
0&= P_f(X_0,X_1,X_2)+
P_f(X_0, -\rho^2X_0-\rho X_1, \rho^2 X_0-\rho^2X_1+X_2)
\\
&+
P_f(X_0, (1-\rho^2)X_0+\rho^2 X_1, (\rho^2-1)X_0 +(1-\rho^2)X_1+X_2)
\\
&+
P_f(X_0, 2X_0 -X_1,-2X_0+2X_1+X_2)
\\
&
+
P_f(X_0, (1-\rho)X_0+\rho X_1, (\rho-1)X_0+(1-\rho)X_1+X_2)
\\
&+
P_f(X_0,-\rho X_0-\rho^2X_1, \rho X_0-\rho X_1+X_2).
\end{align*}


Combining Theorem \ref{t:2} with the monodromy mapping from \S \ref{s:2.1} gives the
\begin{corollary}
Every contractible path in $\pi_1((\mathbb{P}^1\backslash\{0,1,\infty\})^2\backslash\Delta, (\overrightarrow{01},\overrightarrow{01})$ gives a relation on $P_f(X_0,X_1,X_2)$ in the sense that one or more of the equations in Theorem \ref{t:2} can be associated with the path once it is mapped into $PU(2,1;\mathbb{Z}[\rho])$.
\end{corollary}

This is certainly a less satisfactory result than Theorem \ref{t:1}. Although we have an explicit description of $\Gamma_1$ in terms of the paths $\delta_j$ in $(\mathbb{P}^1\backslash\{0,1,\infty\})^2\backslash\Delta$ given by Shiga (see \S \ref{s:1.1} above), the more explicit association of the paths in $\Gamma_{(01,01)}^{(4)}$ with elements of $PU(2,1;\mathbb{Z}[\rho])$ doesn't yield any relations since it is itself a free group, as a subgroup of a free group by Proposition \ref{p:free}. The interesting fact in this connection was mentioned in the introduction: as Ihara explains in \cite{Ihara:braids}, the pentagonal relation on paths in $\mathcal{M}_{0,5}$ runs between 5 specific tangential basepoints which are permuted by the action of a 5-cycle in $S_5$ on the canonical $S_5$-action on the space. Consequently the $S_4$-orbit of any of these tangential basepoints does not include all of the others. For this reason, any pentagonal path in $\mathcal{M}_{0,5}$ which gives rise to the pentagonal relation on $GT$ will not give rise to a five term relation on the Picard period polynomials.

\appendix
\section{Runge's ring of Picard modular forms for $PU(2,1; \mathbb{Z}[\rho])$}
\label{a:Runge}

For the notation and terminology used in this appendix, see \S \ref{s:2.2} above.

In \cite{RungePicardMod}, Runge proved the following theorem:
\begin{appthmARunge}
The ring of Picard modular forms of Picard type $M$
is given by 
\[
\mathbb{C}[P_6^2, P_{12},P_6P_{18},P_{18}^2]
\]
where for $j=6, 12$ and $18$, $P_j$ is a homogeneous polynomial of degree $j$ in the theta constants $f_{\mathbf{k}}(\Omega)$ of (\ref{e:tc}) (with $\mathbf{k} = \mathbf{1},\mathbf{2}$ or $\mathbf{3}$ as above). Together these polynomials generate the ring of invariants of the polynomial ring in the $f_{\mathbf{k}}$ under the action of the group generated by 
\[
\frac{1}{2}
\left[ \begin{array}{ccc}
-1+i\sqrt{3} & 0 & 0
\\
i(1+\sqrt{3}) & -1+i & 1+i 
\\
(2+\sqrt{3})(1-i) & -1+i &-1-i
\end{array}
        \right]
        \]
        and 
        \[\frac{1}{4}\left[ \begin{array}{ccc}
                0& -2+2i\sqrt{3} & 0
                \\
                -1+\sqrt{3}-i-i\sqrt{3} & 1+\sqrt{3}-i+i\sqrt{3} & -1+\sqrt{3}-i-i\sqrt{3}
                \\
                -1-3\sqrt{3}-i-i\sqrt{3} & -1+\sqrt{3}-i-i\sqrt{3} & -3-\sqrt{3}+i+i\sqrt{3} 
                \end{array}\right].      
        \]
\end{appthmARunge}
Yet, Runge never gave the explicit polynomials which appear here. Running Hawes' InvariantRing package \cite{Hawes:M2} on Macaulay2 took almost two weeks to complete on a state-of-the-art machine. For the sake of facilitating explicit computations of at least some of the Picard period polynomials, we reproduce the polynomials $P_6$ and $P_{12}$ below. $P_{18}$ has 190 terms.

\[  P_6 =                                              
(-\frac{23}{8}i\rho - \frac{23}{16}i +\frac{19}{8})f_{\mathbf{1}}^6   
+ (-\frac{31}{4}i\rho -\frac{31}{8}i+7)f_{\mathbf{1}}^5f_{\mathbf{2}}+
(\frac{55}{8}i\rho+\frac{55}{16}i-\frac{55}{8})f_{\mathbf{1}}^4f_{\mathbf{2}}^2
\]
\[
+(\frac{15}{2}i\rho+\frac{15}{4}i-5)f_{\mathbf{1}}^3f_{\mathbf{2}}^3
+(\frac{55}{8}i\rho+\frac{55}{16}i-\frac{55}{8})f_{\mathbf{1}}^2f_{\mathbf{2}}^4
+(-\frac{31}{4}i\rho-\frac{31}{8}i+7)f_{\mathbf{1}}f_{\mathbf{2}}^5
\]
\[
+(-\frac{23}{8 }i\rho -\frac{23}{16}i +\frac{19}{8} )f_{\mathbf{2}}^6
+(-\frac{167}{24}i\rho -\frac{113}{12} -\frac{167}{24}\rho+\frac{59}{24})f_{\mathbf{1}}^5f_{\mathbf{3}}
\]
\[
+(-\frac{235}{24}i\rho-\frac{40}{3}i-\frac{235}{24}\rho+\frac{85}{24})f_{\mathbf{1}}^4f_{\mathbf{2}}f_{\mathbf{3}}  
+(\frac{65}{12}i\rho +\frac{85}{12}i +\frac{65}{12}\rho-\frac{5}{3})f_{\mathbf{1}}^3f_{\mathbf{2}}^2f_{\mathbf{3}}
\]
\[
+ (\frac{65}{12}i\rho + \frac{85}{12}i + \frac{65}{12}\rho - \frac{5}{3})f_{\mathbf{1}}^2 f_{\mathbf{2}}^3 f_{\mathbf{3}} 
+ (- \frac{235}{24}i\rho - \frac{40}{3}i - \frac{235}{24}\rho + \frac{85}{24})f_{\mathbf{1}}f_{\mathbf{2}}^4 f_{\mathbf{3}} 
\]
\[
+ (- \frac{167}{24}i\rho - \frac{113}{12}i - \frac{167}{24}\rho + \frac{59}{24})f_{\mathbf{2}}^5 f_{\mathbf{3}} 
+ (- \frac{155}{16}i - \frac{35}{3}\rho - \frac{35}{6})f_{\mathbf{1}}^4 f_{\mathbf{3}}^2  
+ (- \frac{25}{2}i - \frac{85}{6}\rho - \frac{85}{12})f_{\mathbf{1}}^3 f_{\mathbf{2}}f_{\mathbf{3}}^2  
\]
\[
+ (- \frac{45}{8}i - 5\rho - \frac{5}{2})f_{\mathbf{1}}^2 f_{\mathbf{2}}^2 f_{\mathbf{3}}^2  
+ (- \frac{25}{2}i - \frac{85}{6}\rho - \frac{85}{12})f_{\mathbf{1}}f_{\mathbf{2}}^3 f_{\mathbf{3}}^2  
+ (- \frac{155}{16}i - \frac{35}{3}\rho - \frac{35}{6})f_{\mathbf{2}}^4 f_{\mathbf{3}}^2 
\]
\[ 
+ (\frac{65}{12}i\rho - \frac{5}{3}i - \frac{65}{12}\rho - \frac{85}{12})f_{\mathbf{1}}^3 f_{\mathbf{3}}^3  
+ (\frac{35}{4}i\rho - \frac{15}{4}i - \frac{35}{4}\rho - \frac{25}{2})f_{\mathbf{1}}^2 f_{\mathbf{2}}f_{\mathbf{3}}^3  
\]
\[
+ (\frac{35}{4}i\rho - \frac{15}{4}i - \frac{35}{4}\rho - \frac{25}{2})f_{\mathbf{1}}f_{\mathbf{2}}^2 f_{\mathbf{3}}^3  
+ (\frac{65}{12}i\rho - \frac{5}{3}i - \frac{65}{12}\rho - \frac{85}{12})f_{\mathbf{2}}^3 f_{\mathbf{3}}^3  
+ (\frac{85}{24}i\rho + \frac{85}{48}i - \frac{25}{8})f_{\mathbf{1}}^2 f_{\mathbf{3}}^4  
\]
\[
+ (\frac{85}{12}i\rho + \frac{85}{24}i -\frac{ 25}{4})f_{\mathbf{1}}f_{\mathbf{2}}f_{\mathbf{3}}^4  
+ (\frac{85}{24}i\rho + \frac{85}{48}i - \frac{25}{8})f_{\mathbf{2}}^2 f_{\mathbf{3}}^4  
+ (\frac{7}{8}i\rho + \frac{5}{4}i + \frac{7}{8}\rho - \frac{3}{8})f_{\mathbf{1}}f_{\mathbf{3}}^5  
\]
\[
+ (\frac{7}{8}i\rho + \frac{5}{4}i + \frac{7}{8}\rho - \frac{3}{8})f_{\mathbf{2}}f_{\mathbf{3}}^5
+ (\frac{3}{16}i + \frac{1}{6}\rho + \frac{1}{12})f_{\mathbf{3}}^6
\]

  \[ P_{12}= 
  \left(- \frac{6637}{128}i - \frac{22975}{384}\rho - \frac{22975}{768}\right)f_{\mathbf{1}}^{12}   
  + \left(- \frac{11493}{64}i - \frac{415}{2}\rho - \frac{415}{4}\right)f_{\mathbf{1}}^{11}  f_{\mathbf{2}} 
  \]
  \[
  + \left(\frac{249}{64}i + \frac{294}{64}\rho + \frac{295}{128}\right)f_{\mathbf{1}}^{10}  f_{\mathbf{2}}^2  
  + \left(\frac{82615}{64}i + \frac{8945}{6}\rho + \frac{8945}{12}\right)f_{\mathbf{1}}^9 f_{\mathbf{2}}^3  
  \]
  \[
  + \left(\frac{223005}{128}i + \frac{257445}{128}\rho + \frac{257445}{256}\right)f_{\mathbf{1}}^8 f_{\mathbf{2}}^4  
  + \left(\frac{60711}{32}i + 2190\rho + 1095\right)f_{\mathbf{1}}^7 f_{\mathbf{2}}^5  
  \]
  \[
  + \left(\frac{1743}{32}i + \frac{2065}{32}\rho + \frac{2065}{64}\right)f_{\mathbf{1}}^6 f_{\mathbf{2}}^6  
  + \left(\frac{60711}{32}i + 2190\rho + 1095\right)f_{\mathbf{1}}^5 f_{\mathbf{2}}^7  
  \]
  \[
  + \left(\frac{223005}{128}i + \frac{257445}{128}\rho +\frac{257445}{256}\right)f_{\mathbf{1}}^4 f_{\mathbf{2}}^8  
  + \left(\frac{82615}{64}i + \frac{8945}{6}\rho + \frac{8945}{12}\right)f_{\mathbf{1}}^3 f_{\mathbf{2}}^9  
  \]
  \[
  + \left(\frac{249}{64}i + \frac{295}{64}\rho + \frac{295}{128}\right)f_{\mathbf{1}}^2 f_{\mathbf{2}}^{10}   
  + \left(- \frac{11493}{64}i - \frac{415}{2}\rho - \frac{415}{4}\right)f_{\mathbf{1}}f_{\mathbf{2}}^{11}   
  \]
  \[
  + \left(- \frac{6637}{128}i - \frac{22975}{384}\rho - \frac{22975}{768}\right)f_{\mathbf{2}}^{12}   
  + \left(\frac{15319}{64}i\rho - \frac{5617}{64}i - \frac{15319}{64}\rho - \frac{2617}{8}\right)f_{\mathbf{1}}^{11}  f_{\mathbf{3}} 
  \]
  \[
  + \left(\frac{60161}{64}i\rho - \frac{10999}{32}i - \frac{60161}{64}\rho - \frac{82159}{64}\right)f_{\mathbf{1}}^{10}  f_{\mathbf{2}}f_{\mathbf{3}} 
  \]
  \[
  + \left(- \frac{1955}{64}i\rho + \frac{355}{32}i + \frac{1955}{64}\rho +\frac{2665}{64}\right)f_{\mathbf{1}}^9 f_{\mathbf{2}}^2 f_{\mathbf{3}} 
  \]
  \[
  + \left(- \frac{292245}{64}i\rho + \frac{106935}{64}i + \frac{292245}{64}\rho + \frac{99795}{16}\right)f_{\mathbf{1}}^8 f_{\mathbf{2}}^3 f_{\mathbf{3}} 
  \]
  \[
  + \left(- \frac{238065}{32}i\rho + \frac{87165}{32}i + \frac{238065}{32}\rho + \frac{162615}{16}\right)f_{\mathbf{1}}^7 f_{\mathbf{2}}^4 f_{\mathbf{3}} 
  \]
  \[
  + \left(- \frac{176895}{32}i\rho + \frac{16185}{8}i + \frac{176895}{32}\rho + \frac{241635}{32}\right)f_{\mathbf{1}}^6 f_{\mathbf{2}}^5 f_{\mathbf{3}} 
  \]
  \[
  + \left(- \frac{176895}{32}i\rho + \frac{16185}{8}i + \frac{176895}{32}\rho + \frac{241635}{32}\right)f_{\mathbf{1}}^5 f_{\mathbf{2}}^6 f_{\mathbf{3}} \]
  \[
  + \left(- \frac{238065}{32}i\rho + \frac{87165}{32}i + \frac{238065}{32}\rho + \frac{162615}{16}\right)f_{\mathbf{1}}^4 f_{\mathbf{2}}^7 f_{\mathbf{3}} 
  \]
  \[
  + \left(- \frac{292245}{64}i\rho + \frac{106935}{64}i + \frac{292245}{64}\rho + \frac{99795}{16}\right)f_{\mathbf{1}}^3 f_{\mathbf{2}}^8 f_{\mathbf{3}} 
  \]
  \[
  + \left(- \frac{1955}{64}i\rho + \frac{355}{32}i + \frac{1955}{64}\rho + \frac{2665}{64}\right)f_{\mathbf{1}}^2 f_{\mathbf{2}}^9 f_{\mathbf{3}} 
  \]
  \[
  + \left(\frac{60161}{64}i\rho - \frac{10999}{32}i - \frac{60161}{64}\rho - \frac{82159}{64}\right)f_{\mathbf{1}}f_{\mathbf{2}}^{10}  f_{\mathbf{3}} 
  \]
  \[
  + \left(\frac{15319}{64}i\rho - \frac{5617}{64}i - \frac{15319}{64}\rho - \frac{2617}{8}\right)f_{\mathbf{2}}^{11}  f_{\mathbf{3}} 
  + \left(\frac{7627}{8}i\rho + \frac{7627}{16}i - \frac{105783}{128}\right)f_{\mathbf{1}}^{10}  f_{\mathbf{3}}^2  \]
  \[
  + \left(\frac{29815}{8}i\rho + \frac{29815}{16}i - \frac{206415}{64}\right)f_{\mathbf{1}}^9 f_{\mathbf{2}}f_{\mathbf{3}}^2  
  + \left(- \frac{495}{4}i\rho - \frac{495}{8}i + \frac{13365}{128}\right)f_{\mathbf{1}}^8 f_{\mathbf{2}}^2 f_{\mathbf{3}}^2  
  \]
  \[
  + \left(- \frac{29715}{2}i\rho - \frac{29715}{4}i + \frac{205875}{16}\right)f_{\mathbf{1}}^7 f_{\mathbf{2}}^3 f_{\mathbf{3}}^2  
  + \left(- \frac{181155}{8}i\rho - \frac{181155}{16}i + \frac{1255305}{64}\right)f_{\mathbf{1}}^6 f_{\mathbf{2}}^4 f_{\mathbf{3}}^2  
  \]
  \[
  + \left(- \frac{85473}{4}i\rho - \frac{85473}{8}i + \frac{592011}{32}\right)f_{\mathbf{1}}^5 f_{\mathbf{2}}^5 f_{\mathbf{3}}^2  
  + \left(- \frac{181155}{8}i\rho - \frac{181155}{16}i + \frac{1255305}{64}\right)f_{\mathbf{1}}^4 f_{\mathbf{2}}^6 f_{\mathbf{3}}^2  
  \]
  \[
  + \left(- \frac{29715}{2}i\rho - \frac{29715}{4}i + \frac{205875}{16}\right)f_{\mathbf{1}}^3 f_{\mathbf{2}}^7 f_{\mathbf{3}}^2  
  + \left(- \frac{495}{4}i\rho - \frac{495}{8}i + \frac{13365}{128}\right)f_{\mathbf{1}}^2 f_{\mathbf{2}}^8 f_{\mathbf{3}}^2  
  \]
  \[
  + \left(\frac{29815}{8}i\rho + \frac{29815}{16}i - \frac{206415}{64}\right)f_{\mathbf{1}}f_{\mathbf{2}}^9 f_{\mathbf{3}}^2  
  + \left(\frac{7627}{8}i\rho + \frac{7627}{16}i - \frac{105783}{128}\right)f_{\mathbf{2}}^{10}  f_{\mathbf{3}}^2  
  \]
  \[
  + \left(\frac{229685}{192}i\rho + \frac{313915}{192}i + \frac{229685}{192}\rho - \frac{42115}{96}\right)f_{\mathbf{1}}^9 f_{\mathbf{3}}^3  
  \]
  \[
  + \left(\frac{272295}{64}i\rho + \frac{371805}{64}i + \frac{272295}{64}\rho - \frac{49755}{32}\right)f_{\mathbf{1}}^8 f_{\mathbf{2}}f_{\mathbf{3}}^3  
  \]
  \[
  + \left(- \frac{14415}{16}i\rho - \frac{4905}{4}i - \frac{14415}{16}\rho + \frac{5205}{16}\right)f_{\mathbf{1}}^7 f_{\mathbf{2}}^2 f_{\mathbf{3}}^3  
  \]
  \[
  + \left(- \frac{247975}{16}i\rho - \frac{169405}{8}i - \frac{247875}{16}\rho + \frac{90835}{16}\right)f_{\mathbf{1}}^6 f_{\mathbf{2}}^3 f_{\mathbf{3}}^3  
  \]
  \[
  + \left(- \frac{736875}{32}i\rho - \frac{1006515}{32}i - \frac{736875}{32}\rho + \frac{33705}{4}\right)f_{\mathbf{1}}^5 f_{\mathbf{2}}^4 f_{\mathbf{3}}^3  
  \]
  \[
  + \left(- \frac{736875}{32}i\rho - \frac{1006515}{32}i - \frac{736875}{32}\rho + \frac{33705}{4}\right)f_{\mathbf{1}}^4 f_{\mathbf{2}}^5 f_{\mathbf{3}}^3  
  \]
  \[
  + \left(- \frac{247975}{16}i\rho - \frac{169405}{8}i - \frac{247975}{16}\rho + \frac{90835}{16}\right)f_{\mathbf{1}}^3 f_{\mathbf{2}}^6 f_{\mathbf{3}}^3  
  \]
  \[
  + \left(- \frac{14415}{16}i\rho - \frac{4905}{4}i - \frac{14415}{16}\rho + \frac{5205}{16}\right)f_{\mathbf{1}}^2 f_{\mathbf{2}}^7 f_{\mathbf{3}}^3 
  \]
  \[
   + \left(\frac{272295}{64}i\rho + \frac{371805}{64}i + \frac{272295}{64}\rho - \frac{49755}{32}\right)f_{\mathbf{1}}f_{\mathbf{2}}^8 f_{\mathbf{3}}^3  
   \]
   \[
   + \left(\frac{229685}{192}i\rho + \frac{313915}{192}i + \frac{229685}{192}\rho - \frac{42115}{96}\right)f_{\mathbf{2}}^9 f_{\mathbf{3}}^3  
   \]
   \[
   + \left(\frac{225855}{128}i + \frac{260645}{128}\rho + \frac{260645}{256}\right)f_{\mathbf{1}}^8 f_{\mathbf{3}}^4  
   + \left(\frac{166125}{32}i + \frac{47965}{8}\rho + \frac{47965}{16}\right)f_{\mathbf{1}}^7 f_{\mathbf{2}}f_{\mathbf{3}}^4  
   \]
   \[
   + \left(- \frac{100095}{32}i - \frac{115585}{64}\rho - \frac{115585}{64}\right)f_{\mathbf{1}}^6 f_{\mathbf{2}}^2 f_{\mathbf{3}}^4  
   + \left(- \frac{642825}{32}i - \frac{185585}{8}\rho - \frac{185585}{16}\right)f_{\mathbf{1}}^5 f_{\mathbf{2}}^3 f_{\mathbf{3}}^4  
   \]
   \[
   + \left(- \frac{1732275}{64}i - \frac{2000225}{64}\rho - \frac{2000225}{128}\right)f_{\mathbf{1}}^4 f_{\mathbf{2}}^4 f_{\mathbf{3}}^4 
   + \left(- \frac{642825}{32}i - \frac{185585}{8}\rho - \frac{185585}{16}\right)f_{\mathbf{1}}^3 f_{\mathbf{2}}^5 f_{\mathbf{3}}^4  
   \]
   \[
   + \left(- \frac{100095}{32}i - \frac{115585}{32}\rho - \frac{115585}{64}\right)f_{\mathbf{1}}^2 f_{\mathbf{2}}^6 f_{\mathbf{3}}^4  
   + \left(\frac{166125}{32}i + \frac{47965}{8}\rho + \frac{47965}{16}\right)f_{\mathbf{1}}f_{\mathbf{2}}^7 f_{\mathbf{3}}^4  
   \]
   \[
   + \left(\frac{225855}{128}i + \frac{260645}{128}\rho + \frac{260645}{256}\right)f_{\mathbf{2}}^8 f_{\mathbf{3}}^4  
   + \left(-\frac{38237}{32}i\rho+\frac{14027}{32}i+\frac{38237}{32}+\frac{6533}{4}\right)f_{\mathbf{1}}^7f_{\mathbf{3}}^5  
   \]
   \[
+\left(-\frac{82175}{32}i\rho+\frac{7505}{8}i+\frac{82175}{32}\rho+ \frac{112195}{32}\right)f_{\mathbf{1}}^6f_{\mathbf{2}}f_{\mathbf{3}}^5  
\]
\[               
+\left(\frac{ 115527}{32}i\rho-\frac{21111 }{16}i -\frac{115527}{32}\rho-\frac{157749}{32}\right)f_{\mathbf{1}}^5f_{\mathbf{2}}^2f_{\mathbf{3}}^5 
\]
\[
+\left(\frac{387005  }{32}i\rho- \frac{141665}{32}i-\frac{387005}{32}\rho- \frac{264335}{16}\right)f_{\mathbf{1}}^4f_{\mathbf{2}}^3f_{\mathbf{3}}^5 
\]
\[
+\left(
\frac{ 387005 }{32}i\rho-\frac{141665 }{32}i-\frac{   387005}{32}\rho-   
      \frac{264335}{16}\right)f_{\mathbf{1}}^3f_{\mathbf{2}}^4f_{\mathbf{3}}^5 
      \]
      \[
      +\left(\frac{ 115527}{32}i\rho-\frac{21111}{16}-\frac{115527}{32}\rho-\frac{ 157749}{32}\right)f_{\mathbf{1}}^2f_{\mathbf{2}}^5f_{\mathbf{3}}^5
      \]
      \[
      +\left(-\frac{ 82175}{32}i\rho+\frac{      7505}{8}i+\frac{82175 }{32}\rho+\frac{   112195}{32}\right)f_{\mathbf{1}}f_{\mathbf{2}}^6f_{\mathbf{3}}^5 
      \]
      \[
      +\left(-\frac{38237}{32}i\rho+\frac{14027}{32}i+\frac{38237}{32}\rho+\frac{6533}{4}\right)f_{\mathbf{2}}^7f_{\mathbf{3}}^5
      + \left(-\frac{     7605}{8}i\rho-\frac{      7605}{16}i+\frac{52743}{64}\right)f_{\mathbf{1}}^6 f_{\mathbf{3}}^6 
\]
\[
      +\left(-\frac{     8895 }{8}i\rho-\frac{     8895}{16}i+\frac{ 30747}{32}\right)f_{\mathbf{1}}^5f_{\mathbf{2}}f_{\mathbf{3}}^6 
+\left(\frac{  8115}{2}i\rho+\frac{      8115}{4}i-\frac{    224655}{64}\right)f_{\mathbf{1}}^4 f_{\mathbf{2}}^2 f_{\mathbf{3}}^6 
\]
\[
+\left(
\frac{   31005}{4}i\rho+\frac{      31005}{8}i-\frac{ 107475}{16}\right)f_{\mathbf{1}}^3 f_{\mathbf{2}}^3 f_{\mathbf{3}}^6 
+\left(\frac{   8115}{2}i\rho+\frac{ 8115}{4}i-\frac{224655}{64}\right)f_{\mathbf{1}}^2 f_{\mathbf{2}}^4 f_{\mathbf{3}}^6 
\]
\[
+\left(-\frac{8895}{8}i\rho -
\frac{     8895 }{16}i+\frac{  30747 }{32}\right)f_{\mathbf{1}}f_{\mathbf{2}}^5f_{\mathbf{3}}^6 
\]
\[
+\left(-\frac{7605}{8}i\rho-\frac{7605}{16}i+\frac{    52743}{64}\right)f_{\mathbf{2}}^6 f_{\mathbf{3}}^6   
+\left(-\frac{ 7657 }{32}i\rho -\frac{10463 }{32}i - \frac{7657}{32}\rho+\frac{    1403}{16}\right)  f_{\mathbf{1}}^5 f_{\mathbf{3}}^7 
\]
\[
+\left(\frac{775}{32}i\rho+\frac{1055}{32}i+\frac{775}{32}\rho-\frac{35}{4}\right)f_{\mathbf{1}}^4f_{\mathbf{2}}f_{\mathbf{3}}^7+\left(\frac{20305}{16}i\rho+\frac{6935}{4}i+ \frac{ 20305}{16}\rho -     \frac{7435}{16}\right)f_{\mathbf{1}}^3f_{\mathbf{2}}^2f_{\mathbf{3}}^7 
\]
\[
+\left(\frac{ 20305}{16}i\rho +\frac{6935}{4}i+\frac{ 20305}{16}\rho-\frac{    7435}{16}\right)f_{\mathbf{1}}^2f_{\mathbf{2}}^3f_{\mathbf{3}}^7 
+\left(\frac{775}{32}i\rho+\frac{1055}{32}i+\frac{    775}{32}\rho-\frac{    35}{4}\right)f_{\mathbf{1}}f_{\mathbf{2}}^4f_{\mathbf{3}}^7
\]
\[
+\left(-\frac{ 7657}{32}i\rho-\frac{10463}{32}-\frac{    7657}{32}+\frac{1403}{16}\right)f_{\mathbf{2}}^5f_{\mathbf{3}}^7+\left(-\frac{      6735}{128}i-\frac{    7755}{128}\rho-\frac{    7755}{256}\right)f_{\mathbf{1}}^4f_{\mathbf{3}}^8
\]
\[
+\left(\frac{7995}{64}i+\frac{    1155}{8}\rho+\frac{    1155}{16}\right)f_{\mathbf{1}}^3f_{\mathbf{2}}f_{\mathbf{3}}^8
+\left(\frac{22725}{64}i+\frac{    26235}{64}\rho+\frac{    26235}{128}\right)f_{\mathbf{1}}^2f_{\mathbf{2}}^2f_{\mathbf{3}}^8
\]
\[
+\left(\frac{7995}{64}i+\frac{    1155}{8}\rho+\frac{    1155}{16}\right)f_{\mathbf{1}}f_{\mathbf{2}}^3f_{\mathbf{3}}^8
+\left(-\frac{ 6735}{128}i 
      -\frac{        7755}{128}\rho-\frac{    7755}{256}\right)f_{\mathbf{2}}^4f_{\mathbf{3}}^8
      \]
      \[
      +\left(-\frac{ 215}{192}i\rho+\frac{       65}{192}i+\frac{    215}{192}\rho+\frac{    35}{24}\right)f_{\mathbf{1}}^3f_{\mathbf{3}}^9
      +\left(-\frac{1955}{64}i\rho +    
\frac{      355}{32}i+\frac{    1955 }{64}\rho+\frac{   2665}{64}\right)f_{\mathbf{1}}^2f_{\mathbf{2}}f_{\mathbf{3}}^9
\]
\[
+\left(-\frac{   1955}{64}i\rho+\frac{      355}{32}i+\frac{    1955}{64}\rho+\frac{    2665}{64}\right)f_{\mathbf{1}}f_{\mathbf{2}}^2f_{\mathbf{3}}^9
+\left(-\frac{215}{192}i\rho+\frac{65}{192}i+\frac{215}{192}+\frac{35}{24}\right)f_{\mathbf{2}}^3 f_{\mathbf{3}}^9 
\]
\[
+\left(-\frac{11}{4}i\rho-\frac{11}{8}i+\frac{297}{128}\right)f_{\mathbf{1}}^2f_{\mathbf{3}}^{10} 
+\left(-\frac{11}{2}i\rho-\frac{11}{4}i+
\frac{297}{64}\right)f_{\mathbf{1}}f_{\mathbf{2}}f_{\mathbf{3}}^{10}     
\]
\[
+\left(-\frac{11}{4}i\rho-\frac{11}{8}i+\frac{297}{128}\right)f_{\mathbf{2}}^2f_{\mathbf{3}}^{10}                 
   + \left(- \frac{21}{64}i\rho - \frac{27}{64}i - \frac{21}{64}\rho + \frac{3}{32}\right)f_{\mathbf{1}}f_{\mathbf{3}}^{11}   
   \]
   \[
   + \left(- \frac{21}{64}i\rho - \frac{27}{64}i - \frac{21}{64}\rho + \frac{3}{32}\right)f_{\mathbf{2}}f_{\mathbf{3}}^{11}   
   + \left(- \frac{3}{128}i - \frac{5}{128}\rho -\frac{5}{256}\right)f_{\mathbf{3}}^{12}  
\]

\bibliography{bibliog}{}
\bibliographystyle{amsalpha}
\end{document}